\documentclass{amsart}                   

 \usepackage{mathptmx}      
 
\usepackage{amsmath,amsfonts,amssymb,amscd}

 \usepackage{newtxmath}
 
 \usepackage{hyperref}

\newtheorem{thm}{Theorem}
\newtheorem{prop}[thm]{Proposition}
\newtheorem{cor}[thm]{Corollary}
\newtheorem{lem}[thm]{Lemma}
\newtheorem{rem}[thm]{Remark}
\newtheorem{hyp}{Hypothesis}

\newcommand{\be}{\begin{equation}}  \newcommand{\ee}{\end{equation}}
\newcommand{\bea}{\begin{eqnarray*}}  \newcommand{\eea}{\end{eqnarray*}}

\numberwithin{equation}{section}

\begin{document}

\newcommand{\C}{{\mathbb C}}

\newcommand\R{{\mathbb R}}
\newcommand\Z{{\mathbb Z}}
\newcommand\N{{\mathbb N}}

\newcommand\F{{\mathcal F}}
\newcommand\T{{\mathcal T}}
\newcommand\E{{\mathbb E}}

\renewcommand\L{{\Lambda}}
\newcommand\M{{\mathcal M}}
\renewcommand\P{{\mathcal P}}
\renewcommand\S{{\mathcal S}}
\newcommand\cR{{\mathcal R}}

\newcommand\D{{\mathcal D}}
\newcommand\Q{{\mathcal Q}}

\renewcommand\l{\lambda}
\renewcommand\O{\Omega}

\newcommand\el{\,\stackrel{law}{=}\,}

\newcommand\tr{\,\mbox{\rm tr}\,}

\newcommand\etr{\, \mbox{\rm etr}\,}

\title{Interacting diffusions on positive definite matrices }

\author{Neil O'Connell    
}

\address{
              School of Mathematics and Statistics, University College Dublin, Dublin 4, Ireland }

\maketitle

\begin{abstract}
We consider systems of Brownian particles in the space of positive definite matrices, 
which evolve independently apart from some simple interactions.  We give examples of such processes which 
have an integrable structure.  These are related to $K$-Bessel functions 
of matrix argument and multivariate generalisations of these functions.  
The latter are eigenfunctions of a particular quantisation of the non-Abelian Toda lattice.
\end{abstract}

\setcounter{tocdepth}{1}

\section{Introduction}
\label{intro}

In recent years, there has been much progress in the development of integrable models
in probability, particularly interacting particle systems related to representation theory
and integrable systems.  A well-known example is the coupled system of SDE's
\be\label{sd}
dx_1=d\beta_1,\qquad dx_i=d\beta_i+e^{x_{i-1}-x_i} dt, \ i=2,\ldots,N,
\ee
where $\beta_i$ are independent one-dimensional Brownian motions.
This process is closely related to the Toda lattice and has been extensively 
studied~\cite{bc,bcf,is,noc12,oy1,sv,spohn}.

It is natural to consider non-commutative generalisations of such processes.
In this paper, we consider some interacting systems of Brownian particles in 
the space $\P$ of positive $n\times n$ real symmetric matrices.  
One of the main examples we consider is a generalisation of the system \eqref{sd}, 
a diffusion process in $\P^N$ with infinitesimal generator
\be\label{Tintro}
T=\Delta_{X_1}+\sum_{i=2}^N [\Delta_{X_i} + 2 \tr(X_{i-1}\partial_{X_i})],
\ee
where $\Delta_X$ denotes the Laplacian, and $\partial_X$ denotes the
partial matrix derivative, on $\P$.
In the case $n=1$ with $x_i=\ln X_i$, it is equivalent to the system \eqref{sd}.  
We will show that this process is
related to a quantisation of the non-Abelian 
Toda lattice in~$\P$.

In another direction, Matsumoto and Yor~\cite{my} obtained an analogue of Pitman's $2M-X$ 
theorem for exponential functions of Brownian motion, which is closely related to Dufresne's identity~\cite{duf}.
These results were recently extended to the matrix setting by Rider and Valk\'o~\cite{rv16} 
(see also Bougerol~\cite{bougerol} for related results in the complex case).
We discuss this example in some detail,
and give a new proof (and slight generalisation) of Rider and Valk\'o's result.  We also consider
an example of a pair of interacting Brownian particles in $\P$ with a `reflecting wall'.

The outline of the paper is as follows.  In the next section we present some preliminary
background material.  This is followed, in Sections 3--6, by a series of examples with small numbers of
particles.  In Section 7, we discuss the example \eqref{Tintro} and
its relation to a quantisation of the non-Abelian Toda lattice.  In Section 8, we outline how this
example is related to some B\"acklund transformations
for the classical non-Abelian Toda lattice.  In Section 9, we briefly discuss a related class
of processes in which the underlying motion of particles is not governed by the Laplacian,
but rather a related diffusion process which was introduced and studied 
by Norris, Rogers and Williams~\cite{nrw}.
In Section 10, we briefly outline how the framework developed in this paper 
extends to the complex setting, with particular emphasis on some extensions
of the (Hermitian) Matrix Dufresne identity of~\cite{rv16}.

\section{Preliminaries}

We mostly follow the nomenclature of Terras~\cite{terras},
to which we refer the reader for more background.
Let $\P$ denote the space of positive $n\times n$ real symmetric matrices.
For $a\in GL(n)$ and $X\in\P$, write $X[a]=a^t X a$.  This defines an action of $GL(n)$ on $\P$.

For $X\in\P$, we will use the notation $$|X|=\det(X),\qquad \etr(X)=\exp(\tr(X)),$$
and denote by $X^{1/2}$ the unique positive square root of $X$.

\subsection{Differential operators}

The partial  derivative on $\P$ is defined, writing $X=(x_{ij})$, by
$$\partial_X = \left(\frac12(1+\delta_{ij})\frac{\partial}{\partial x_{ij}} \right)_{1\le i,j \le n}.$$
We define the Laplacian on $\P$ by $\Delta_X = \tr \vartheta_X^2$,
where $\vartheta_X = X \partial_X$.  
The Laplacian is a $GL(n)$-invariant differential operator on $\P$, meaning
$(\Delta f)^a=\Delta f^a$ for all $a\in GL(n)$, where $f^a(X)=f(X[a])$.
In fact, the differential operators $L_k=\tr(\vartheta_X)^k,\ k=1,2,\ldots$ are all 
$GL(n)$-invariant~\cite[Exercise 1.1.27]{terras}.  This follows from the fact
that, if $Y=X[a]$ for some fixed $a\in GL(n)$, then 
\be\label{cr1} \partial_X=a\, \partial_Y \, a^t.\ee
If $Y=X^{-1}$, then $\vartheta_Y=-\vartheta_X'$, where $\vartheta_X' f = (\partial_X f)X$.
It follows that $L_k^Y=(-1)^k L_k^X$.  In particular, the Laplacian is invariant under this
change of variable.

\subsection{Chain rule for quadratic change of variables}

It is known \cite{mj,bfj} that for each $X,Y\in\P$, the equation $Y=AXA$ has a unique 
solution in $\P$, namely
$$A=X^{-1/2}(X^{1/2}YX^{1/2})^{1/2} X^{-1/2}.$$
This fact may also be deduced from
 \cite[Theorem 4.1]{wonham}.
If $X$ is fixed, then
\be\label{cr2}
\partial_A=XA\, \partial_Y+\partial_Y AX.
\ee

\subsection{Calculus}

For $A\in\P$ fixed:
\be\label{dax}
\vartheta_X \tr(AX) = XA,\qquad \vartheta_X \tr(AX^{-1}) = -AX^{-1}
\ee
\be\label{d2ax}
\Delta_X \tr (AX)= \frac{n+1}{2} \tr(AX),\qquad  \Delta_X \ \tr (AX^{-1})= \frac{n+1}{2} \tr(AX^{-1})
\ee
\be\label{d2F}
\Delta_X \ F = \left[\Delta_X \ln F + \tr(\vartheta_X\ln F)^2\right] F
\ee
\be\label{d2etr}
\Delta_X \etr(AX) = \left[ \frac{n+1}{2} \tr(AX) + \tr(AXAX)\right] \etr(AX)
\ee
\be\label{pr}
\Delta(fg)=(\Delta f)g+f (\Delta g)+2\tr[(\vartheta_Xf)( \vartheta_X g)]
\ee
\be
\vartheta_X (AX) = \frac12 (A+\tr A) X
\ee
For positive integers $k$,
\be\label{txk}
\vartheta_X \tr X^k=k X^k,\qquad \vartheta_X^2 \tr X^k=k \vartheta_X(X^k)
= \frac{k^2}2 X^k+ \frac{k}2 \sum_{j=1}^{k} X^j \tr X^{k-j}
\ee
\be\label{power}
\Delta_X \tr X^k = \frac{k^2}2 \tr X^k + \frac{k}2 \sum_{j=1}^{k} \tr X^j \tr X^{k-j}
\ee
\be\label{power2}
\Delta_X \tr X^{-k} = \frac{k^2}2 \tr X^{-k}+ \frac{k}2 \sum_{j=1}^{k} \tr X^{-j} \tr X^{j-k}
\ee

\subsection{Integration}
Denote the $GL(n)$-invariant volume element on $\P$ by 
$$\mu(dX)=|X|^{-(n+1)/2}\prod_{1\le i\le j\le n} dx_{ij},$$
where $dx_{ij}$ is the Lebesgue measure on $\R$.
If we write $X=a[k]$, where $a$ is diagonal with entries $a_1,\ldots,a_n>0$
and $k\in O(n)$, then
\be\label{ed}
\mu(dX)=c_n dk \prod_{i<j} |a_i-a_j| \prod_{i=1}^n a_i^{-(n-1)/2-1} da_i ,
\ee
where $dk$ is the normalised Haar measure on $O(n)$ and $c_n$ is a normalisation constant.

\subsection{Power and Gamma functions}
For $s\in\C^n$, define the power function 
$$p_s(X)=\prod_{k=1}^n |X^{(k)}|^{s_k},$$
where $X^{(k)}$ denotes the $k\times k$ upper left hand corner of $X$.
For $s\in\C^n$ satisfying 
\be\label{gamma} 2\Re (s_k+\ldots+s_n)>k-1,\qquad k=1,\ldots,n\ee
define
$$\Gamma_n(s)=\int_\P p_s(X) \etr(- X) \mu(dX)
 = \pi^{n(n-1)/4} \prod_{k=1}^n \Gamma\left(s_k+\ldots+s_n-\frac{k-1}2\right).$$
For $s=(0,\ldots,0,\nu)$, we will write $p_s(X)=e_\nu(X)=|X|^\nu$ and  $\Gamma_n(\nu)=\Gamma_n(0,\ldots,0,\nu)$.
The spherical functions on $\P$ are defined, for $s\in\C^n$, by
$$h_s(X)=\int_{O(n)} p_s(X[k])\; dk,$$
where $dk$ denotes the normalised Haar measure on $O(n)$.

The power function $p_s$ is an eigenfunction of the 
Laplacian on $\P$, with eigenvalue
$$\l_2(s)=\sum_{i=1}^n r_i^2+\frac{n-n^3}{48},\qquad
r_i=s_i+\cdots+s_n+\frac{n+1-2i}4.$$
(See, for example, \cite[Exercise 1.2.12 \& Equation (1.93)]{terras}.)
The functions $p_{s,k}(X)=p_s(X[k]),\ k\in O(n)$, and $h_s(X)$, 
are also eigenfunctions
of $\Delta$ with eigenvalue $\l_2(s)$.

For $s=(0,\ldots,0,\nu)$, so that $p_s(X)=e_\nu(X)=|X|^\nu$, 
we note:
\be\label{de}
\vartheta_X e_\nu(X) = \nu \, e_\nu(X)\, I_n,\qquad \tr \vartheta_X \, e_\nu(X) = n\nu \, e_\nu(X), 
\qquad \tr \vartheta_X \, \ln e_\nu(X) = n\nu ,\ee
\be\label{d2e} \Delta_X \, e_\nu(X) = n\nu^2 \, e_\nu(X),\qquad \Delta_X \, \ln e_\nu(X) = 0. \ee

\subsection{Bessel functions}

Following Terras~\cite{terras}, for $s\in\C^n$ and $V,W\in\P$, define
$$K_n(s|V,W)=\int_\P p_s(Y) \etr(- VY -  WY^{-1}) \mu(dY).$$
As remarked in~\cite{terras}, following Exercise~1.2.16, we can
always reduce one of the positive matrix arguments in $K_n(s|V,W)$
to the identity: if $W=I[g]$, where $g$ is upper triangular 
with positive diagonal coefficients, then
\be\label{inv}
K_n(s| V,W) = p_s(W) K_n(s|gVg^t,I).
\ee
In the following, if $s=(0,\ldots,0,\nu)$ we will write $K_n(s|V,W)=K_n(\nu|V,W)$.

For $\nu\in\C$ and $X\in\P$, define 
\be\label{herz}
B_\nu(X)=\int_\P e_\nu(Y) \etr(- XY - Y^{-1}) \mu(dY).
\ee
This function was introduced by Herz~\cite{herz}, and is related to $K_n$ by
$$B_\nu(X)=K_n(\nu| X,I).$$ 
We note that $B_{-\nu}(X)=e_\nu(X)B_\nu(X)$, which implies
\be\label{sym}
e_\nu(V) K_n(\nu|V,W)=e_\nu(W) K_n(-\nu|V,W).
\ee
The asymptotic behaviour of $B_\nu(X)$ for large arguments has been studied via Laplace's method in the paper \cite{bw},
see also \cite[Appendix B]{g18}.  If we fix $M\in\P$ and let $X=z^2 M^2/2$, then it holds that, as $z\to\infty$,
$$B_\nu(X)=C(\nu,M)
z^{-\nu-n(n+1)/4} e^{-z\tr M} (1+O(z^{-1})),$$
where
$$C(\nu,M)=2^{\nu-n(n+1)/4} \pi^{n(n+1)/4} |M|^{-\nu-1/2} \prod_{i<j}(m_i+m_j)^{-1/2},$$
and $m_i$ denote the eigenvalues of $M$.  In particular, taking $M=I$ and $z^2/2=\alpha$, say, we deduce
from the same application of Laplace's method, the following lemma, which we record here for later reference.
\begin{lem}\label{gig}
For $\alpha>0$, let $A(\alpha)$ be distributed according to the (matrix GIG) law
$$B_\nu(\alpha I)^{-1} \etr(-\alpha A-A^{-1}) \mu(dA).$$
Then $\alpha^{1/2} A(\alpha)\to I$, in probability, as $\alpha\to\infty$.
\end{lem}

\subsection{Standard probability distributions}\label{spd}

The {\em Wishart distribution} on $\P$ with parameters $\Sigma\in\P$ and $p>n-1$ has density
$$\Gamma_n(p/2)^{-1} |\Sigma^{-1}X/2|^{p/2} \etr(-\Sigma^{-1}X/2) \mu(dX).$$
If $p\ge n$ is an integer and $A$ is a $p\times n$ random matrix with independent standard normal entries,
then $\Sigma^{1/2}A^tA\Sigma^{1/2}$ is Wishart distributed with parameters $\Sigma$ and $p$.
The {\em inverse Wishart distribution} on $\P$, with parameters $\Sigma\in\P$ and $p>n-1$,
is the law of the inverse of a Wishart matrix with parameters $\Sigma$ and $p$,
and has density
$$\Gamma_n(p/2)^{-1} |\Sigma^{-1}X^{-1}/2|^{p/2} \etr(-\Sigma^{-1}X^{-1}/2) \mu(dX).$$
The {\em matrix GIG} (generalised inverse Gaussian) distribution on $\P$ with parameters $\nu\in\R$ and $A,B\in\P$
is defined by
$$K_n(\nu| A,B)^{-1} |X|^{\nu} \etr(-AX-BX^{-1}) \mu(dX).$$

\subsection{Invariant kernels}

A kernel $k(X,Y)$ is {\em invariant} if, for all $a\in GL(n)$,
$$k(X[a],Y[a])=k(X,Y).$$  
For $k$ sufficiently smooth, this implies that
$$\Delta_X k(X,Y)=\Delta_Y k(X,Y).$$
Examples include 
\be\label{k}
k(X,Y)=\etr(-YX^{-1})
\ee
and, more generally, for $\nu\in\C$,
\be\label{knu}
k_\nu(X,Y)=e_{-\nu}(YX^{-1}) \etr(-YX^{-1}).
\ee
We note that the kernel $k$ defined by \eqref{k} satisfies
\be\label{dk}
\vartheta_X \ln k =  - \vartheta_Y \ln k = YX^{-1} ,\ee
\be\label{d2k} \Delta_X k = \Delta_Y k = \left[ - \frac{n+1}2 \tr(YX^{-1}) + \tr(YX^{-1}YX^{-1})\right] k .\ee

\subsection{Brownian motion and diffusion}
We define Brownian motion in $\P$ with drift $\nu\in\R$ to be the diffusion process in $\P$ with generator
$$\Delta_X^{(\nu)} = \Delta_X + 2\nu\, \tr \vartheta_{X}.$$
More generally, if $\varphi$ is a positive eigenfunction of the Laplacian on $\P$ with eigenvalue $\l$, 
then we may consider the corresponding Doob transform
$$\Delta_X^{(\varphi)} = \varphi(X)^{-1} \circ (\Delta_X-\l) \circ \varphi(X) =
\Delta_X+2\tr(\vartheta_X\ln\varphi(X) \ \vartheta_X).$$
If $\varphi(X)=|X|^\nu$ for some $\nu\in\R$, then $\l=\nu^2 n$ and $\Delta_X^{(\varphi)} \equiv \Delta_X^{(\nu)}$.
We shall refer to the diffusion process with infinitesimal generator $\Delta_X^{(\varphi)}$
as a Brownian motion in $\P$ with drift $\varphi$.

A Brownian motion in $\P$, with drift $\nu$, may be constructed as follows.
Let $b_t,\ t\ge 0$ be a standard Brownian motion in the Lie algebra $\mathfrak{gl}(n,\R)$
of real $n\times n$ matrices, that is, each matrix entry evolves as a standard Brownian
motion on the real line.  Set $\beta_t=b_t/\sqrt{2}+\nu t$.  Define a Markov process $G_t,\ t\ge 0$ 
in $GL(n)$ via the Stratonovich SDE: $\partial G_t = \partial\beta_t\, G_t$.  
When $\nu=0$, this is called a right-invariant Brownian motion in $GL(n)$; 
thus, we shall refer to $G$ as a right-invariant Brownian motion in $GL(n)$ with drift $\nu$.  
Then $Y=G^tG$ is a Brownian motion in $\P$ with drift $\nu$.  Note that $Y$ satisfies
the Stratonovich SDE
$$\partial Y = G^t (\partial\beta+\partial\beta^t) G.$$
By orthogonal invariance of the underlying Brownian motion in $\mathfrak{gl}(n,\R)$, one may replace the $G$ 
and $G^t$ factors in this equation by $Y^{1/2}$ to obtain a 
closed SDE for the evolution of $Y$.

We will also consider more general diffusions on $\P^r$,
with generators of the form
$$L=\sum_{i=1}^r [\Delta_{X_i} + \tr(a_i(X)\partial_{X_i})],$$
where the $a_i$ are locally Lipschitz functions on $\P^r$.  
For such generators, we may take the domain to be
$C^2_c(\P^r)$, the set of continuously twice
differentiable, compactly supported, functions on $\P^r$.
If $\rho$ is a probability measure on $\P^r$ and the martingale
problem associated with $(L,\rho)$ is well posed, then we
may construct a realisation of the corresponding Markov
process by solving the (Stratonovich) SDE's:
$$\partial X_i = X_i^{1/2} \partial S_i X_i^{1/2} + a_i(X) dt,$$
where $b_1,\ldots,b_r$ are independent
standard Brownian motions in $\mathfrak{gl}(n,\R)$,
$S_i=(b_i+b_i^t)/\sqrt{2}$, and $X(0)$ is chosen according to $\rho$.

\section{Brownian particles with one-sided interaction}\label{onesided}
Consider the differential operator on $\P^2$ defined  by
\be\label{2p} T = \Delta_Y + \Delta_X + 2\tr(Y\partial_X).\ee
Let $k(X,Y)=\etr(-YX^{-1})$ and consider the integral operator defined,
for suitable test functions $f$ on $\P^2$, by
$$(K f)(X) = \int_\P k(X,Y) f(X,Y) \mu(dY),\qquad X\in\P.$$
Then, on a suitable domain, the following intertwining relation holds:
\be\label{int-burke}
\Delta \circ K = K \circ T.
\ee
Indeed, let us write $k=k(X,Y)$, $f=f(X,Y)$ and note the following identities:
$$\Delta_X k=\Delta_Y k,\qquad \vartheta_X k=YX^{-1}k,$$
$$\Delta_X(kf) = f \Delta_X k+k\Delta_X f+2\tr(\vartheta_Xk\vartheta_Xf)
= f \Delta_Y k+k\Delta_X f+2k\tr(Y\partial_X f).$$
It follows, using the fact that $\Delta$ is self-adjoint with respect to $\mu$, that
\begin{eqnarray*}
\Delta (K f) (X) &=& \Delta_X \int_\P k f \mu(dY) \\
&=& \int_\P (f\Delta_Y k+k\Delta_X f+2k\tr(Y\partial_X f)) \mu(dY)\\
&=& \int_\P k \; (Tf)\; \mu(dY),
\end{eqnarray*}
as required.  

Now suppose $\varphi$ is a positive eigenfunction of $\Delta$ with eigenvalue $\l$
such that
$$\tilde\varphi(X)=\int_\P \varphi(Y) k(X,Y) \mu(dY)<\infty .$$
Then $\tilde\varphi$ is also a positive eigenfunction of $\Delta$ with eigenvalue $\l$:
\bea
\Delta_X \tilde\varphi(X) &=& \int_\P \varphi(Y) \Delta_X k(X,Y) \mu(dY) \\
&=&  \int_\P \varphi(Y) \Delta_Y k(X,Y) \mu(dY) \\
&=&  \int_\P  [\Delta_Y \varphi(Y)] k(X,Y) \mu(dY) \\
&=&   \l \int_\P \varphi(Y) k(X,Y) \mu(dY) = \l \tilde\varphi(X).
\eea
For example, if $\varphi=p_s$, for some $s\in\R^n$ satisfying \eqref{gamma}, 
then $\tilde\varphi=\Gamma_n(s)\varphi$ (see, for example, \cite[Exercise 1.2.4]{terras}).
Similarly, if $s\in\R^n$ satisfies \eqref{gamma},
and $\varphi=h_s$ or $\varphi=p_{s,k}$ for some $k\in O(n)$, then it also holds 
that $\tilde\varphi=\Gamma_n(s)\varphi$.  More generally, $\tilde \varphi$ is 
a constant multiple of $\varphi$ whenever $\varphi$ is a simultaneous 
eigenfunction of the Laplacian and the integral operator with kernel $k(X,Y)$;
note that these two operators commute, since $\Delta_X k = \Delta_Y k$.

Define
$$(K_\varphi f)(X) = \tilde\varphi(X)^{-1} \int_\P \varphi(Y) k(X,Y) f(X,Y) \mu(dY),$$
and
$$T^{(\varphi)}=\varphi(Y)^{-1}\circ (T-\l) \circ \varphi(Y)=\Delta^{(\varphi)}_Y + \Delta_X + 2\tr(Y\partial_X).$$
Then \eqref{int-burke} extends to:
\be\label{int-burke2}
\Delta^{(\tilde\varphi)} \circ K_\varphi = K_\varphi \circ T^{(\varphi)}.
\ee
The intertwining relation \eqref{int-burke2} has a probabilistic interpretation, as follows.
Set $$\pi(X,Y)=\tilde\varphi(X)^{-1} \varphi(Y) k(X,Y).$$
Let $\rho$ be a probability measure on $\P$ and define a probability
measure on $\P\times\P$ by
$$\sigma(dX,dY)=\pi(X,Y) \rho(dX) \mu(dY).$$
Suppose that $\varphi$ is such that the martingale problems associated with $(\Delta^{(\tilde\varphi)},\rho)$ 
and $(T^{(\varphi)},\sigma)$ are well-posed,
and that $(X_t,Y_t)$ is a diffusion process with infinitesimal generator $T^{(\varphi)}$ and initial law $\sigma$.
Then it follows from the theory of Markov functions
(see Appendix~\ref{mfa}) that, with respect to its own filtration, 
$X_t$ is a Brownian motion with drift $\tilde\varphi$ and initial distribution $\rho$;
moreover, the conditional law of $Y_t$, given $X_s,\ s\le t$, only depends on $X_t$
and is given by $ \pi(X_t,Y) \mu(dY)$.  
This statement is analogous to the Burke output theorem for the $M/M/1$ queue, although
in this context the `output' (a Brownian motion with drift $\tilde\varphi$) need not have the 
same law as the `input' (a Brownian motion with drift $\varphi$).  Note however that these 
Brownian motions are equivalent whenever $\tilde\varphi$ is a constant multiple of $\varphi$,
so whenever this holds the output does have the same law as the input.  This is always the
case when $n=1$, as was observed in the paper~\cite{oy1}.

We note that the intertwining relation \eqref{int-burke2} also implies that 
$$(T^{(\varphi)})^*\pi=0,$$
where $(T^{(\varphi)})^*$ is the formal adjoint of $T^{(\varphi)}$.

One can replace $k$ by any invariant kernel $k'$ and the above remains valid with
$$T=\Delta_Y + \Delta_X + 2\tr(\vartheta_X\ln k'(X,Y)\ \vartheta_X).$$
For example, if $k'=k_{\nu}$, defined by \eqref{knu}, then
$$T = \Delta_Y + \Delta^{(\nu)}_X + 2\tr(Y\partial_X).$$
In this case, we require $\tilde\varphi= k_\nu\varphi$ to be finite, where
$$k_\nu\varphi(X)=\int_\P \varphi(Y) k_\nu(X,Y) \mu(dY).$$
For example, if $\varphi(X)=|X|^\l$, then this holds provided $2(\l-\nu)>n-1$,
in which case $k_\nu\varphi=\Gamma_n(\l-\nu)\varphi$.  For this example,
the associated martingale problems are well posed, 
as shown in Appendix~\ref{wpmp} (Example 2),
so we may state the following theorem.  

For $2a>n-1$, we define the Markov kernel 
\be\label{pia}
\Pi_a(X,dY)=\Gamma_n(a)^{-1} |YX^{-1}|^a \etr(-YX^{-1}) \mu(dY).
\ee
\begin{thm} \label{ln-burke} 
Suppose $2(\l-\nu)>n-1$, and let $(X_t,Y_t)$ be a diffusion process in $\P^2$ with infinitesimal generator 
$$T_{\l,\nu}=\Delta_Y^{(\l)} + \Delta^{(\nu)}_X + 2\tr(Y\partial_X),$$ 
and initial law $\delta_{X_0}(dX) \Pi_{\l-\nu}(X,dY)$.
Then, with respect to its own filtration,  $X_t$ is a Brownian motion with drift $\l$ started at $X_0$.  
Moreover, the conditional law of $Y_t$, given $X_s,\ s\le t$, only depends on $X_t$
and is given by $ \Pi_{\l-\nu}(X_t,dY)$. 
\end{thm}

The above example extends naturally to a system of $N$ particles with one-sided interactions, as follows.
Let $\nu_2,\ldots,\nu_N\in\R$, and $\varphi$ a positive eigenfunction of $\Delta$ such that
$$\tilde\varphi(X)=(k_{\nu_N}\circ\cdots\circ k_{\nu_2})\varphi(X)<\infty.$$
For example, if $\varphi(X)=|X|^{\nu_1}$ then this condition is satisfied provided $\nu_i<\nu_1$
for all $1<i\le N$, in which case we have
$$\tilde\varphi(X)=\prod_{i=2}^N \Gamma_n(\nu_1-\nu_i) \ |X|^{\nu_1}.$$
Define
\be\label{t}
T=\Delta_{X_1}^{(\varphi)}+\sum_{i=2}^N [\Delta_{X_i}^{(\nu_i)} + 2 \tr(X_{i-1}\partial_{X_i})],
\ee
$$\pi(X_1,\ldots,X_N)=\tilde\varphi(X_N)^{-1} \varphi(X_1) \prod_{i=2}^N k_{\nu_i}(X_i,X_{i-1}),$$
$$(Kf)(X_N)=\int_{\P^{N-1}} \pi(X_1,\ldots,X_{N-1},X_N) \,
f(X_1,\ldots, X_N) \,\mu_{N-1}(dX_1,\ldots, dX_{N-1}).$$
Then 
$$\Delta_{X_N}^{(\tilde\varphi)} \circ K = K \circ T.$$
This implies that $T^*\pi=0$ and, moreover, if $\varphi$ is such that the relevant martingale
problems are well posed and the system is started in equilibrium, 
then $X_N$ is a Brownian motion, in its own filtration, with drift $\tilde\varphi$.
This certainly holds in the case $\varphi(X)=|X|^{\nu_1}$, with $\nu_i<\nu_1$ for all $1<i\le N$.
Note that this can also be seen as a direct consequence of Theorem~\ref{ln-burke}.

Finally we remark that, by a simple change of variables, one may also consider
\be\label{tprime}
T'=\Delta_{X_1}^{(\varphi)}+\sum_{i=2}^N [\Delta_{X_i}^{(\nu_i)} - 2 \tr(X_i X_{i-1}^{-1}\vartheta_{X_i})].
\ee
The operator $T'$ is related to $T$ as follows.  Write $T=T_X(\varphi,\nu)$,
$T'=T'_X(\varphi,\nu)$.  Then, under the change of variables $Y_i=X_i^{-1}$, 
$T_X(\varphi,\nu)=T'_Y(\bar\varphi,\bar\nu)$, where $\bar\varphi(Y)=\varphi(X^{-1})$
and $\bar\nu_i=-\nu_i$.

In this case, if we assume that $\tilde\varphi = (k_{-\nu_N}\circ\cdots\circ k_{-\nu_2})\varphi$ is finite,
and define
$$\pi'(X_1,\ldots,X_N)=\tilde\varphi(X_N)^{-1} \varphi(X_1) \prod_{i=2}^N k_{-\nu_i}(X_{i-1},X_{i}),$$
$$(K'f)(X_N)=\int_{\P^{N-1}} \pi'(X_1,\ldots,X_{N-1},X_N) \,
f(X_1,\ldots, X_N) \,\mu_{N-1}(dX_1,\ldots, dX_{N-1}),$$
then it holds that
$\Delta_{X_N}^{(\tilde\varphi)} \circ K' = K' \circ T'$,
with the analogous conclusions.

\section{Connection with Bessel functions}

The previous example, with two particles, extends naturally to
$$G=\Delta_Y + \Delta_{X_1} +2\tr(Y\partial_{X_1})+ \Delta_{X_2}-2\tr(X_2Y^{-1}X_2\partial_{X_2}).$$
Note that this is a combination of the $T$ and $T'$ of the previous section.

Writing $X=(X_1,X_2)$, define
$$H=\Delta_{X_1}+\Delta_{X_2}-2\tr(X_1^{-1}X_2),\qquad
q(X,Y)=  \etr(-YX_1^{-1}-X_2 Y^{-1}).$$
Consider the integral operator, defined for suitable $f$ on $\P^2\times\P$ by
$$(Qf)(X)=\int_\P q(X,Y) f(X,Y) \mu(dY),\qquad X\in\P^2.$$
Then the following intertwining relation holds:
\be\label{int-sym}
H \circ Q = Q\circ G.
\ee
Indeed, let us write $q=q(X,Y)$, $f=f(X,Y)$ and note that
$$H_Xq=\Delta_Y q,\qquad
H_X(qf)=f \Delta_Y q+q(G-\Delta_Y)f.$$
The claim follows, using the fact that $\Delta$
is self-adjoint with respect to $\mu$.

Suppose that $\varphi$ is a positive eigenfunction of $\Delta$ with eigenvalue $\l$
such that
$$\psi(X)=\int_\P \varphi(Y) q(X,Y) \mu(dY)<\infty .$$
Then $\psi$ is a positive eigenfunction of $H$ with eigenvalue $\l$.
We remark that, if $\varphi=p_s$, then
$$\psi(X)= K_n(s|\ X_1^{-1},X_2).$$
Let us define
$$(Q_\varphi f)(X) = \psi(X)^{-1} \int_\P \varphi(Y) q(X,Y) f(X,Y) \mu(dY),$$
$$H^{(\psi)}=\psi(X)^{-1}\circ (H-\l) \circ \psi(X),$$
$$G^{(\varphi)}=\varphi(Y)^{-1}\circ (G-\l) \circ \varphi(Y)=\Delta^{(\varphi)}_Y + \Delta_{X_1} +2\tr(Y\partial_{X_1})+ \Delta_{X_2}-2\tr(X_2Y^{-1}X_2\partial_{X_2}).$$
Then \eqref{int-sym} extends to:
\be\label{int-sym2}
H^{(\psi)} \circ Q_\varphi = Q_\varphi \circ G^{(\varphi)}.
\ee
This intertwining relation has a probabilistic interpretation, as follows.
Let $\rho$ be a probability measure on $\P^2$ and define a probability
measure on $\P^2\times\P$ by
$$\sigma(dX,dY)=\psi(X)^{-1} \rho(dX) \varphi(Y) q(X,Y) \mu(dY).$$
Suppose that $\varphi$ is such that the martingale problems associated with $(H^{(\psi)},\rho)$ 
and $(G^{(\varphi)},\sigma)$ are well-posed, and
that $(X_t,Y_t)$ is a diffusion process with infinitesimal generator $G^{(\varphi)}$ and initial law $\sigma$.
Then, with respect to its own filtration, $X_t$ is a diffusion with generator $H^{(\psi)}$ and initial distribution $\rho$;
moreover, the conditional law of $Y_t$, given $X_s,\ s\le t$, only depends on $X_t$
and is given by
$$\psi(X_t)^{-1} \varphi(Y) q(X_t,Y) \mu(dY).$$

The above example is a special case of a more general construction which will be discussed in
Section \ref{pN}.

\section{Matrix Dufresne identity and $2M-X$ theorem}\label{my}

Let $M=\Delta_Y +\tr(Y\partial_A)$.  If $Y=AXA$ then, in the variables $(X,A)$, we can write
$$M=\Delta_X -2\tr(XAX\partial_X)+\tr(AXA\partial_A).$$
To see this, let $f=f(X,A)=g(AXA,A)$ and first note that, by invariance, 
$$\Delta_Xf=\Delta_Y g(Y,A)\ \Big|_{Y=AXA}.$$
Let us write $g_1(Y,A)=\partial_Y g(Y,A)$, $g_2(Y,A)=\partial_A g(Y,A)$.
By \eqref{cr1} and \eqref{cr2}, 
$$\partial_X f = A g_1(AXA,A) A,\qquad \partial_A f = XA g_1(AXA,A) + g_1(AXA,A) AX.$$
It follows that
$$\tr(AXA\partial_A)f-2\tr(XAX\partial_X)f = \tr(AXAg_2(AXA,A)) = \tr(Y\partial_A) g(Y,A) \ \Big|_{Y=AXA},$$
as required.

Let us define
$$J=\Delta_X-\tr X,\qquad
p(X,A)=\etr(-AX-A^{-1})$$
and the corresponding integral operator
$$(Pf)(X)=\int_\P p(X,A) f(X,A) \mu(dA).$$
Then, on a suitable domain, the following intertwining relation holds:
\be\label{int-my}
J \circ P = P\circ M.
\ee
Indeed, let us write $p=p(X,A)$, $f=f(X,A)$ and first note that
$$\Delta_X(pf)=f\Delta_Xp+p(\Delta_X-2\tr(XAX\partial_X))f.$$
Now, using the fact that $\Delta$ is self-adjoint with respect to $\mu$,
together with the identity $$\Delta_X\etr(-AX)=\Delta_A\etr(-AX),$$ 
we have
\begin{eqnarray*}
\int_\P f (\Delta_Xp) \mu(dA)&=&\int_\P \etr(-A^{-1}) f\ [\Delta_A \etr(-AX)] \mu(dA)\\
&=& - \int_\P \tr(\vartheta_A\etr(-AX)\ \vartheta_A(\etr(-A^{-1}) f))\mu(dA) \\
&=& - \int_\P p \tr(-AX \vartheta_A-AXA^{-1})f \mu(dA) .
\end{eqnarray*}
It follows that $J(Pf)(X)=P(Mf)(X)$, as required.

Note that the intertwining relation \eqref{int-my} implies
\be\label{int-duf}
J \circ D = D \circ \Delta ,
\ee
where $D$ is the linear operator defined, for suitable $f:\P\to\C$ by
$$(Df)(X)=\int_\P f(AXA) \etr(-AX-A^{-1}) \mu(dA).$$
The intertwining relation \eqref{int-duf} is essentially equivalent to \cite[Corollary 6]{rv16}.

Now suppose $\varphi$ is a positive eigenfunction of $\Delta$ with eigenvalue $\l$
such that $\beta=D\varphi<\infty$.
Then it follows from \eqref{int-duf} that $\beta$ is a positive eigenfunction of $J$ with eigenvalue $\l$.
Note that if we write $\beta(X)=\varphi(X) B_\varphi(X)$,
then this implies
\be\label{FK}
(\Delta_X^{(\varphi)}-\tr X)B_\varphi(X)=0.
\ee
This suggests that, for suitable $\varphi$, the function $B_\varphi$ 
admits a natural probabilistic interpretation, via the Feynman-Kac formula,
and this is indeed the case.
\begin{prop}\label{fk-my} Let $\varphi$ be a positive eigenfunction of $\Delta$
such that $D\varphi<\infty$, and the martingale problem associated with
$\Delta^{(\varphi)}$ is well posed for any initial condition in $\P$.
Let $Y$ be a Brownian motion in $\P$ with drift $\varphi$ started at $X$, and 
denote by $\E_X$ the corresponding expectation.  Assume that, for any $X\in\P$, 
\be\label{finite}
Z=\int_0^\infty \tr Y_s\ ds<\infty
\ee
almost surely, and define $M_\varphi(X)=\E_X e^{-Z}$.
Suppose that $\lim_{X\to 0} M_\varphi(X) =1$ and 
\be\label{BC}
\qquad \lim_{X\to 0} B_\varphi(X) = C_\varphi ,
\ee
where $C_\varphi>0$ is a constant.  Then $B_\varphi(X)=C_\varphi\ M_\varphi(X)$
and, moreover, $B_\varphi$ is the unique bounded solution to \eqref{FK} satisfying 
the boundary condition \eqref{BC}.
\end{prop}
\begin{proof}
It follows from the Feynman-Kac formula that $M_\varphi$ satisfies
$$(\Delta_X^{(\varphi)}-\tr X)M_\varphi(X)=0.$$
To prove uniqueness, up to a constant factor, suppose $U(X)$ is another bounded solution which
vanishes as $X\to 0$.  Note that, by \eqref{finite}, it must hold that $Y_t\to 0$ almost surely as $t\to\infty$.
Thus, 
$$U(Y_t) \exp\left(-\int_0^t \tr Y_s\ ds\right)$$
is a bounded martingale which converges to $0$ almost surely, as $t\to\infty$,
hence must be identically zero almost surely, which implies $U= 0$, as required.
\end{proof}

If $\varphi(X)=|X|^{-\nu/2}$, then $B_\varphi=B_{-\nu}$ is the matrix $K$-Bessel function
defined by \eqref{herz}.  
If we denote the eigenvalues of $Y_t$ by $\l_i(Y_t)$, arranged in decreasing order, 
then, as shown in~\cite{rv16}, it holds almost surely that, for any initial condition $X\in\P$,
$$\lim_{t\to\infty} \frac1{t} \log \l_i(Y_t) = - \nu + (n-2i+1)/2.$$
In particular, if $\nu>(n-1)/2$, then \eqref{finite} holds.  
In this example, the process $Y_t$ is $GL(n,\R)$-invariant, so we may write
$$M_\varphi(X)=\E_X \exp\left(-\int_0^\infty \tr Y_s\ ds\right)=
\E_I \exp\left(-\int_0^\infty \tr (XY_s)\ ds\right),$$
and it follows, using bounded convergence, that $\lim_{X\to 0} M_\varphi(X) =1$.
On the other hand, again using bounded convergence, we have
$$\lim_{X\to 0} B_{-\nu}(X) = \Gamma_n(\nu).$$
Putting this together and applying Proposition~\ref{fk-my} yields
the following conclusion, in agreement with \cite[Theorem 2]{rv16}.
When $n=1$, this is Dufresne's identity~\cite{duf}.  

\begin{cor}\label{mdr} If $Y$ is a Brownian motion in $\P$ with drift $-\nu/2$,
started at the identity, then $\int_0^\infty Y_s\ ds$ is inverse Wishart 
distributed with parameters $I/2$ and $2\nu$. \end{cor}

More generally, suppose $\varphi=h_s$, where $s\in\R^n$. 
Define new variables $r_i$ by
$$2(s_i+\cdots+s_n)=2r_i+i-\frac{n+1}2.$$
It is well known that the spherical function $h_s$ is invariant
under permutations of the $r_i$, so we may assume that
$r_1>\cdots>r_n$.  Then it may be shown \cite{bpc}, via a straightforward
modification of the proof of the second part of Theorem 3.1 in \cite{bj}, 
that $$\lim_{t\to\infty} \frac1{2t} \log \l_i(t) = r_i,$$
almost surely, for any initial condition.
In particular, \eqref{finite} holds if, and only if, $r_1<0$.
This condition also ensures that
\be
\int_\P h_s(AXA)\etr(-A^{-1})\mu(dA)<\infty ,
\ee
for all $X\in\P$.
Then, by the uniqueness property of the spherical functions on $\P$
with a given set of eigenvalues \cite[Proposition 1.2.4]{terras}, 
\be\label{id}
\int_\P h_s(AXA)\etr(-A^{-1})\mu(dA)=c_s h_s(X),
\ee
where
\be\label{cs}
c_s=\int_\P h_s(A^2)\etr(-A^{-1})\mu(dA) .
\ee
This implies that
\be
B_s(X)=h_s(X)^{-1} \int_\P h_s(AXA) \etr(-AX-A^{-1}) \mu(dA)
\ee
is bounded.  Using the homogeneity property $h_s(cX)=c^dh_s(X)$, 
where $d=\sum_k ks_k$, we see that, for any fixed $X\in\P$,
$$\lim_{c\to 0} \E_{cX} \exp\left(-\int_0^\infty \tr Y_s\ ds\right) =
\lim_{c\to 0} \E_{X} \exp\left(- c\int_0^\infty \tr Y_s\ ds\right) = 1$$
and
$$\lim_{c\to 0} B_s(cX) = \lim_{c\to 0} h_s(X)^{-1} \int_\P h_s(AXA) \etr(-cAX-A^{-1}) \mu(dA) = c_s.$$
Assuming that these limits extend to
\be\label{limits}
\lim_{X\to 0} \E_X \exp\left(-\int_0^\infty \tr Y_s\ ds\right) =1,\qquad \lim_{X\to 0} B_s(X)=c_s,
\ee
Proposition~\ref{fk-my} would then imply that
\be\label{fk-s}
\E_X \exp\left(-\int_0^\infty \tr Y_s\ ds\right) = c_s^{-1} B_s(X).
\ee
Again using the homogeneity property of $h_s$, this is equivalent to the identity:
\be\label{GD}
\int_0^\infty \tr Y_s\ ds \el \tr (AX),
\ee where $A$ is distributed according to 
the probability measure
$$\nu_s(dA)=c_s^{-1} h_s(X)^{-1} h_s(AXA)\etr(-A^{-1})\mu(dA) .$$
To make this claim rigorous, one would need to establish the existence of the limits 
in \eqref{limits}.  We will not pursue this here.

We remark that, writing $r=-\mu$, we may compute, for $n=1,2,3$:
\be\label{csf}
c_s=\prod_i\Gamma(2\mu_i)\prod_{i<j} B(1/2,\mu_i+\mu_j),
\ee
where $B(x,y)=\Gamma(x)\Gamma(y)/\Gamma(x+y)$ is the beta function.
The analogue of this formula in the complex case is given by \eqref{cl} below,
which is valid for all $n$.  It seems natural to expect \eqref{csf} to be valid for all $n$ also.

Returning to the general setting, let us define
$$(P_\varphi f)(X) = \beta(X)^{-1} \int_\P \varphi(AXA) p(X,A) f(X,A) \mu(dA),$$
$$J^{(\beta)}=\beta(X)^{-1}\circ (J-\l) \circ \beta(X),$$
$$M^{(\varphi)}=\varphi(Y)^{-1}\circ (M-\l) \circ \varphi(Y)=\Delta^{(\varphi)}_Y +\tr(Y\partial_A).$$
As before, with the change of variables $Y=AXA$, we can also write
$$M^{(\varphi)}=\Delta_X^{(\varphi)} -2\tr(XAX\partial_X)+\tr(AXA\partial_A).$$
Then \eqref{int-my} extends to:
\be\label{int-my2}
J^{(\beta)} \circ P_\varphi = P_\varphi \circ M^{(\varphi)}.
\ee
This intertwining relation has a probabilistic interpretation, as follows.
Let $\rho$ be a probability measure on $\P$ and define a probability
measure on $\P\times\P$ by
$$\sigma(dX,dA)= \rho(dX) \gamma_X(dA),$$
where
$$\gamma_X(dA)=\beta(X)^{-1} \varphi(AXA) p(X,A) \mu(dA).$$
Suppose that $\varphi$ is such that the martingale problems associated with $(J^{(\beta)},\rho)$ 
and $(M^{(\varphi)},\sigma)$ are well-posed, and
that $(X_t,A_t)$ is a diffusion process with infinitesimal generator $M^{(\varphi)}$ and initial law $\sigma$.
Then we may apply Theorem~\ref{mf} to conclude that,
with respect to its own filtration, $X_t$ is a diffusion with generator $J^{(\beta)}$ 
and initial distribution $\rho$;
moreover, the conditional law of $A_t$, given $X_s,\ s\le t$, only depends on $X_t$
and is given by $\gamma_{X_t}(dA)$.  

These conditions certainly hold when $\varphi(X)=|X|^{\nu/2}$,
for any $\nu\in\R$, in which case we obtain the following generalisation of  \cite[Proposition 23]{rv16}.
Define $\beta_\nu(X)=|X|^{\nu/2} B_\nu(X)$.

\begin{thm}\label{mmy}  Let $Y_t,\ t\ge 0$ be a Brownian motion in $\P$ with drift $\nu/2$ started at $I$,
and let $A_t=\int_0^t Y_s ds$.  Fix $X_0\in\P$, choose $\tilde A_0$ at random, independent of $Y$, according
to the distribution $\gamma_{X_0}(dA)$, and define 
$$\tilde Y_t=\tilde A_0 X_0^{1/2} Y_t X_0^{1/2} \tilde A_0,\qquad \tilde A_t = \tilde A_0 + \int_0^t \tilde Y_s\ ds.$$ 
Then $X_t=\tilde A_t^{-1} \tilde Y_t \tilde A_t^{-1},\ t\ge 0$ is a diffusion in $\P$ with infinitesimal generator
$$L_\nu=\Delta_X+2\tr(\vartheta_X\ln \beta_\nu(X)\ \vartheta_X),$$
started at $X_0$.
In particular, as a degenerate case, the process 
$A_t^{-1} Y_t A_t^{-1},\ t>0$ is a diffusion in $\P$ with infinitesimal generator $L_\nu$.
\end{thm}
\begin{proof} The relevant martingale problems are well posed, as shown in Appendix \ref{wpmp} (Example 5),
so the first claim follows from Theorem~\ref{mf}, as outlined above.
For the second, we can let $X_0=mI$ and consider the limit as $m\to\infty$.  By Lemma~\ref{gig}, 
$m^{1/2} \tilde A_0\to I$ in probability, as required.
\end{proof}

The second statement was proved, under the condition $2|\nu|>n-1$, by Rider and Valk\'o~\cite{rv16}.  Related results in the complex setting have been obtained 
by Bougerol~\cite{bougerol}.
In the case $n=1$, the above theorem is due to Matsumoto and Yor~\cite{my}, see also Baudoin~\cite{baudoin}.
We note that, as observed in~\cite{rv16}, the law of the process with generator $L_\nu$ is invariant under
a change of sign of the underlying drift $\nu$, since
$\beta_\nu=\beta_{-\nu}$, cf. \eqref{sym}.

More generally, if $\varphi$ is such that, as $X^{-1}\to 0$,
the measure $\gamma_X(dA)$ is concentrated around $AXA=I$,
and the relevant martingale problems are well-posed, then the corresponding statement should hold:  
if $Y_t$ is a Brownian motion in $\P$ with drift $\varphi$ and $A_t=\int_0^t Y_s ds$,
then $A_t^{-1} Y_t A_t^{-1},\ t>0$ is a diffusion in $\P$ with generator $J^{(\beta)}$.  

\section{Two particles with one-sided interaction and a `reflecting wall'}
Let $\nu\in\R$, and define
$$R_Q=\Delta_Q^{(\nu/2)} +\tr\partial_Q,\qquad N=R_Q   + \Delta_X^{(-\nu/2)} +2\tr(Q\partial_X).$$
We first note that $R$ is self-adjoint with respect to the measure
\be\label{rbm}
\pi(dQ)=|Q|^{\nu}\etr(-Q^{-1})\mu(dQ).
\ee
Define
$$S_X = \Delta_X^{(-\nu/2)} -\tr X^{-1},\qquad C f(X)=\int_\P \etr(-QX^{-1}) f(X,Q) \pi(dQ).$$
Then $R_Q \etr(-QX^{-1}) = S_X \etr(-QX^{-1})$, which implies $S \circ C = C \circ N$.

Now suppose $\rho$ is a positive eigenfunction of $R$ with eigenvalue $\l$ such that
$$\gamma(X)=\int_\P \rho(Q) \etr(-QX^{-1}) \pi(dQ)<\infty .$$
Then $\gamma$ is a positive eigenfunction of $S$ with eigenvalue $\l$.
Define
$$S^{(\gamma)} = \gamma(X)^{-1}\circ (S-\l) \circ \gamma(X),$$
$$R^{(\rho)}=\rho(Q)^{-1}\circ (R-\l) \circ \rho(Q),\qquad
N^{(\rho)}=\rho(Q)^{-1}\circ (N-\l) \circ \rho(Q),$$
$$(C_\rho f)(X) = \gamma(X)^{-1} \int_\P \rho(Q) \etr(-QX^{-1}) f(X,Q) \pi(dQ).$$
Then the above intertwining relation extends to
$$S^{(\gamma)} \circ C_\rho = C_\rho \circ N^{(\rho)}.$$
For example, if $\rho=1$, then $\gamma(X)=B_{\nu}(X^{-1})$.
We note that $S$ is related to the $J$ of the previous section,
via $S_X=J_Y^{(\varphi)}$, where $Y=X^{-1}$ and $\varphi(Y)=|Y|^{\nu/2}$.

These intertwining relations (and their probabilistic interpretations) may be viewed
as non-commutative, `positive temperature' versions of the following well known relation 
between reflecting Brownian motion and the three-dimensional
Bessel process: for appropriate initial conditions, a Brownian motion, reflected off
a [Brownian motion reflected at zero], is a three-dimensional Bessel process 
(\cite{C}, \cite[Prop. 3.5]{aow19}).  In the above, note that $\partial_Q=Q^{-1}\vartheta_Q$
and $Q\partial_X=QX^{-1}\vartheta_X$, so we may interpret the $Q$-process as a 
Brownian motion in $\P$ with `soft reflection off the identity', and the $X$-process as 
a second Brownian motion in $\P$ with `soft reflection off $Q$'. 

\section{Whittaker functions and related processes}\label{pN}

\subsection{Whittaker functions of several matrix arguments}\label{energy}

For $X=(X_1,\ldots,X_N)\in\P^N$ and $\nu\in\C$, we define
$$e_\nu(X) = \prod_{i=1}^N e_\nu(X_i)$$
and, for $X=(X_1,\ldots,X_N)\in\P^N$ and $\l\in\C^N$, 
$$e_\l(X)=\prod_{i=1}^N e_{\l_i}(X_i).$$
For $N\ge 1$, we define the product measure
$$\mu_N(dX)=\mu(dX_1)\ldots\mu(dX_N).$$

Let $\T=\P\times\P^2\times\cdots\times\P^N$ and, for $X\in\P^N$, 
denote by $\T(X)$ the set of $Y=(Y^1,\ldots,Y^N)\in\T$ such that $Y^N=X$.
For $Y\in\T$, define
\be\label{F}
\F(Y)=\sum_{1\le i\le m<N} \tr[Y^m_i(Y^{m+1}_i)^{-1}]+\tr[Y^{m+1}_{i+1}(Y^m_i)^{-1}].
\ee
For $Y\in\T$ and $\l\in\C^N$, define
$$e_\l(Y)=e_{\l_1}(Y^1) \prod_{2\le m\le N} e_{\l_m} (Y^m) \,e_{-\l_m}(Y^{m-1}).$$
\begin{prop}\label{wc}
Let $X\in\P^N$ and $\l\in\C^N$.
\begin{itemize}
\item[(i)] The following integral converges:
$$\psi_\l(X)=\int_{\T(X)} e_\l(Y) e^{-\F(Y)} \prod_{1\le m<N} \mu_m(dY^m).$$
\item[(ii)] The integrand $e_\l(Y) e^{-\F(Y)}$ vanishes as $Y\to\partial\T(X)$.
\end{itemize}
\end{prop}
The proof is given in Appendix \ref{awc}.
We note that, when $N=2$, 
$$\psi_\l(X)=e_{\l_2}(X) K_n(\l_1-\l_2|\ X_1^{-1},X_2).$$
The following properties are straightforward to verify from the definition.
\begin{prop}\label{wmp}
Let $X\in\P^N$, $a\in GL(n)$, $\l\in\C^N$, $\nu\in\C$ and write
$\l'_i=\l_i+\nu$ and $X'=(X_N^{-1},\ldots,X_1^{-1})$.  Then
$$\psi_\l(X_1[a],\ldots,X_N [a])=|a^ta|^{\sum_i\l_i}\psi_\l(X),$$
$$\psi_{\l'}(X)=e_\nu(X)\psi_\l(X),\qquad \psi_\l(X)=\psi_{-\l}(X').$$
\end{prop}
We also anticipate that $\psi_\l(X)$ is symmetric in the parameters $\l_1,\ldots,\l_N$;
in the case $N=2$, this symmetry holds and follows from \eqref{sym}.

\subsection{Interpretation as eigenfunctions}

Consider the differential operator 
\be\label{hnct}
H^{(N)}=\sum_{i=1}^N \Delta_{X_i} - V(X),\qquad
V(X)=2 \sum_{i=1}^{N-1} \tr(X_i^{-1}X_{i+1}).
\ee
This is a quantisation of the $N$-particle non-Abelian Toda lattice on $\P$.

For $\nu\in\C$ and $(X,Y)\in \P^N\times\P^{N-1}$ define
\be\label{kf}
Q^{(N)}_{\nu}(X,Y) = e_{\nu}(X)\, e_{-\nu}(Y)\, \etr\left(-
\sum_{i=1}^{N-1}  (Y_i X_i^{-1}+X_{i+1}Y_i^{-1}) \right) .
\ee
We identify $Q^{(N)}_\nu$ with the integral operator defined, for appropriate $f$, by
$$Q^{(N)}_\nu f(X)=\int_{\P^{N-1}} Q^{(N)}_\nu(X,Y) \,f(Y) \,\mu_{N-1}(dY).$$
Note that, for $\l\in\C^N$, 
\be\label{rec}
\psi^{(N)}_\l=Q^{(N)}_{\l_N} \psi^{(N-1)}_{\l_1,\ldots,\l_{N-1}},
\qquad\psi^{(1)}_\l(X)=|X|^\l.
\ee
It is straightforward to show that
$$(H^{(N)}_X-n\nu^2) \; Q^{(N)}_\nu(X,Y)=H^{(N-1)}_Y \; Q^{(N)}_\nu(X,Y),$$
with the convention $H^{(1)}=\Delta$.  
It follows that, on a suitable domain,
\be\label{int-nct}
(H^{(N)}-n\nu^2) \circ Q^{(N)}_\nu = Q^{(N)}_\nu\circ H^{(N-1)} .
\ee
For $\l\in\C^N$, set 
\be\label{hl}
H_\l = H^{(N)} - \sum_{i=1}^N n\l_i^2.
\ee
The intertwining relation \eqref{int-nct} implies that, for any $\l\in\C^N$, 
\be\label{ev-eqn}
H_\l \psi_\l = 0.
\ee
The integral representation of Proposition~\ref{wc} is a generalisation of the Givental-type formula obtained 
in~\cite{gklo} for the eigenfunctions of the quantum Toda lattice, also known as $GL(n)$-Whittaker functions.
We remark that a slightly richer family of eigenfunctions of $H^{(N)}$ can be obtained by taking
$\psi_\l^{(1)}$ to be an arbitrary eigenfunction of $\Delta$ in the recursive definition \eqref{rec},
provided the corresponding integrals converge.  

\subsection{Feynman-Kac interpretation}

Define Brownian motion in $\P^N$ with drift $\l\in\R^N$ to be the diffusion process in $\P^N$ 
with infinitesimal generator
$$\Delta_\l = \sum_{i=1}^N \Delta_{X_i}^{(\l_i)}.$$
We begin with a lemma.
\begin{lem}\label{duf}
let $Y$ be a Brownian motion in $\P^2$ with drift $\l$ started at $X$,
with $\nu=\l_1-\l_2>(n-1)/2$.  Then 
$$\int_0^\infty \tr[Y_1(t)^{-1} Y_2(t)] dt \el \tr(AW^{-1}) ,$$
where $A=X_1^{-1/2} X_2 X_1^{-1/2}$ and $W$ is a Wishart random matrix with parameters $I$ and $2\nu$.
\end{lem}
This follows from the matrix Dufresne identity~\cite[Theorem 1]{rv16}
(Corollary~\ref{mdr} in the present paper), together with the fact that the 
eigenvalue process of $Y_1(t)^{-1/2} Y_2(t) Y_1(t)^{-1/2}$ has the same law as that of a 
Brownian motion in $\P$ with generator $2\Delta-2\nu\, \tr \vartheta_{X}$, started at $A$.
A proof of the latter claim is given in Appendix~\ref{duf-proof}.

Now let $Y(t),\ t\ge 0$ be a Brownian motion in $\P^N$ with drift $\l\in\R^N$ started at $X$.
Suppose that $\l_i-\l_j>(n-1)/2$ for all $i<j$, and define
$$\varphi_\l(X)= \E \exp\left(-2\sum_{i=1}^{N-1} \int_0^\infty \tr[Y_i(t)^{-1} Y_{i+1}(t)] dt \right) .$$
\begin{prop}\label{fk}
Suppose that $\l_i-\l_j>(n-1)/2$ for all $i<j$.  Then
$$\psi_\l(X)=\prod_{i<j} \Gamma_n(\l_i-\l_j) \ e_\l(X) \varphi_\l(X).$$
Moreover, under these hypotheses, $\psi_\l(X)$ is the unique solution to
\eqref{ev-eqn}, up to a constant factor, such that $e_{-\l}(X) \psi_\l(X)$ is bounded. 
\end{prop}
\begin{proof}
Define
$$\varphi_\l(X)= \E \exp\left(-2\sum_{i=1}^{N-1} \int_0^\infty \tr[Y_i(t)^{-1} Y_{i+1}(t)] dt \right) .$$
It follows from Lemma~\ref{duf} that
$$\lim_{V(X)\to 0} \varphi_\l(X)= 1.$$
By Feynman-Kac, $(\Delta_\l-V)\varphi_\l=0$,
hence $f_\l=e_\l\varphi_\l$ satisfies the eigenvalue equation \eqref{ev-eqn}.
A standard martingale argument (as in the proof of \ref{fk-my}) then shows that $f_\l$ is the unique solution 
to \eqref{ev-eqn} such that $e_{-\l} f_\l$ is bounded and 
$$\lim_{X\to+\infty} e_{-\l} (X)f_\l(X)= 1.$$
It therefore suffices to show that $e_{-\l} \psi_\l$ is bounded and
$$\lim_{V(X)\to 0} e_{-\l}(X) \psi_\l(X) = \prod_{i<j} \Gamma_n(\l_i-\l_j).$$
We prove these statements by induction on $N$, using the recursion \eqref{rec}.

For $N=1$, the claim holds since $\psi_\l(X)=e_\l(X)$ in this case.
For $N\ge 2$, we have
$$\psi^{(N)}_\l(X)=\int_{\P^{N-1}} Q^{(N)}_{\l_N}(X,Y) \psi^{(N-1)}_{\l_1,\ldots,\l_{N-1}}(Y)
\mu_{N-1}(dY),$$
where
$$Q^{(N)}_{\nu}(X,Y) = e_{\nu}(X)\, e_{-\nu}(Y)\, \etr\left(-
\sum_{i=1}^{N-1} (Y_i X_i^{-1}+X_{i+1}Y_i^{-1}) \right) .$$
Let us write 
$$g^{(N)}_\l(X)=e_{-\l}(X) \psi^{(N)}_\l(X)$$
and
$$R_\l(Y)=e_{\l_1-\l_N,\ldots,\l_{N-1}-\l_N}(Y)\, \etr\left(-\sum_{i=1}^{N-1}
\left( Y_i +  A_i Y_i^{-1}\right) \right) ,$$
where $A_i=X_i^{-1/2} X_{i+1} X_i^{-1/2}$.
Changing variables from $Y_i$ to $X_i^{-1/2} Y_i X_i^{-1/2} $, we can write
$$g^{(N)}_\l(X) =\int_{\P^{N-1}} R_\l(Y) g^{(N-1)}_{\l_1,\ldots,\l_{N-1}}(Y') \mu_{N-1}(dY),$$
where $Y'_i=X_i^{1/2} Y_i X_i^{1/2}$.
By induction on $N$, we see immediately that
$$g^{(N)}_\l(X)\le \prod_{i<j} \Gamma_r(\l_i-\l_j).$$
Here we are using
$$R_\l(Y) \le e_{\l_1-\l_N,\ldots,\l_{N-1}-\l_N}(Y)\, \etr\left(-\sum_{i=1}^{N-1}  Y_i \right)  ,$$
which implies
$$\int_{\P^{N-1}} R_\l(Y) \mu_{N-1}(dY) 
\le \prod_{i=1}^{N-1} \Gamma_n(\l_i-\l_N).$$
Now observe that, for each $Y\in\P^{N-1}$, if $V(X)\to 0$ then $V^{(N-1)}(Y')\to 0$.
Thus the claim follows, again by induction, using the dominated convergence theorem.
\end{proof}

In the scalar case $n=1$, the above proposition is a special case of \cite[Proposition 5.1]{boc11},
see also~\cite[Corollary 3]{noc14}.

\subsection{Whittaker measures on $\P^N$}

The following generalises an integral identity due to Stade~\cite{stade}.
The proof is straightforward, by induction, using \eqref{rec}.
Denote by $e_n$ the unit vector $(0,0,\ldots,0,1)$ in $\C^n$.
\begin{prop}\label{stade} Let $s\in\C^n$, $A\in\P$ and $\l,\nu\in\C^N$.
Set $a=\sum_i(\l_i+\nu_i)$.  Assume that $\Re(\l_i+\nu_j)>(n-1)/2$
for all $i$ and $j$, and $\Re(s_k+\cdots+s_n+a)>(k-1)/2$ for all $k$.
Then
$$\int_{\P^N} p_s(X_1) \etr\left(-A^{-1}X_1\right)
 \psi_\l(X)\psi_\nu(X) \mu_N(dX)
= \frac{\Gamma_n(s+ae_n)}{\Gamma_n(a)} p_{s+ae_n}(A) \prod_{i,j} \Gamma_n(\l_i+\nu_j).$$
\end{prop}
In particular, if $\l,\nu\in\R^N$ satisfy $\l_i+\nu_j>(n-1)/2$ for all $i,j$, then
$$W_{\l,\nu}(dX)
=\prod_{i,j} \Gamma_n(\l_i+\nu_j)^{-1} \etr(-X^{-1}_N) \psi_{-\l}(X)\psi_{-\nu}(X) \mu_N(dX)$$
is a probability measure on $\P^N$ which generalises the Whittaker measures of~\cite{cosz,osz}.
Here we have made the change of variables $X'=(X_N^{-1},\ldots,X_1^{-1})$,
as in Proposition~\ref{wmp}.  We remark that, by Proposition~\ref{stade}, for all $B$ such that $I+B\in\P$, we have
$$\int_{\P^N} \etr(-BX_N^{-1}) W_{\l,\nu}(dX) = |I+B|^{-a}.$$
This implies that the $N^{th}$ marginal of $W_{\l,\nu}$ is the inverse Wishart distribution
with parameters $\Sigma=I/2$ and $p=2a$, as defined in Section \ref{spd}.

\subsection{Triangular processes}

Consider the differential operator on $\T$ defined for $\l\in\C^N$ by
\be\label{G}
G_\l = \sum_{1\le i\le m\le N} \Delta_{Y^m_i}^{(\l_i)} 
+ 2\tr[( \epsilon_{im} Y^{m-1}_i (Y^m_i)^{-1}- \epsilon_{1i} Y^m_i (Y^{m-1}_{i-1})^{-1}) \vartheta_{Y^m_i}] ,\ee
where $\epsilon_{ij}=1-\delta_{ij}$.

Define a kernel from $\P^N$ to $\T$ by
$$\Sigma_\l(X,dY)= \delta_X(dY^N)  e_\l(Y) e^{-\F(Y)} \prod_{1\le m<N} \mu_m(dY^m) ,$$
and note that
$$\psi_\l(X) = \int_\T \Sigma_\l(X,dY) .$$
Then the following intertwining relation holds:
\be\label{int-hg}
H_\l \circ \Sigma_\l = \Sigma_\l \circ G_\l .
\ee
This extends \eqref{int-nct} and is readily verified by induction.
Define
\be\label{L}
L_\l = \psi_\l(X)^{-1} \circ H_\l  \circ \psi_\l(X).
\ee
As shown in Appendix~\ref{wpmp} (Example 6), for any $\l\in\R^N$, the martingale problems associated
with $G_\l$ and $L_\l$ are well posed, for any initial conditions.  We thus deduce from the
intertwining relation \eqref{int-hg}, and Theorem~\ref{mf}, the following theorem.

 \begin{thm}\label{QNCT}
Let $\l\in\R^N$ and $X\in\P^N$, and suppose that $Y(t),t\ge 0$ is a diffusion in $\T$ 
with generator $G_\l$ and initial law 
$$\sigma_\l(X,dY) = \psi_\l(X)^{-1} \Sigma_\l(X,dY).$$
Then $Y^N(t),\ t\ge 0$ is a diffusion in $\P^N$ 
with generator $L_\l$.
Moreover, for each $t\ge 0$, the conditional law of 
$Y(t)$, given $Y^N(s),\ s\le t$, is $\sigma_\l(Y^N(t),dY)$.
\end{thm}

\begin{rem} Theorem~\ref{QNCT} is a matrix generalisation of Proposition 9.1 in \cite{noc12}.
It is not clear whether Theorem 3.1 in \cite{noc12} admits a similar generalisation, other 
than in the case $N=2$, where it is given by the statement of Theorem~\ref{mmy}. \end{rem}

We note that $G_\l$ contains the autonomous
$$T_\l = \Delta_\l
+ 2\sum_{i=2}^N \tr[X_{i-1} \partial_{X_i}],\qquad
T'_\l=\Delta_\l
- 2\sum_{i=2}^N \tr[X_i X_{i-1}^{-1} \vartheta_{X_i}].$$
These are special cases of the diffusions with one-sided interactions discussed 
in Section \ref{onesided},
see \eqref{t} and \eqref{tprime}.
It follows from Theorem~\ref{QNCT} that, if $X(t),\ t\ge 0$ is a diffusion in $\P^N$ with generator $T_\l$ (resp. $T'_\l$),
with appropriate initial conditions, then $X_N(t),\ t\ge 0$ is distributed as the first (resp. last) coordinate of a diffusion in $\P^N$ 
with generator $L_\l$.  In the scalar case, this yields very precise information about the law of $X_N(t)$, for special
initial conditions, see for example \cite[Corollary 4.1]{noc12}.  We note however that this application in the scalar case
relies on the Plancherel theory for the quantum Toda lattice, currently unavailable for its non-Abelian generalisation \eqref{hnct}.
In the scalar case, for a particular (singular) initial condition, the random variable $X_N(t)$ may be interpreted as the
logarithmic partition function of the semi-discrete directed polymer in a Brownian environment introduced in~\cite{oy1}.
Unfortunately this interpretation does not extend to the non-Abelian setting, other than in an intuitive sense,
because the SDE's do not admit an explicit solution as an integral in this case.  Nevertheless, to understand the law of $X_N(t)$ 
when $n>1$, and also its asymptotic behaviour as $N,t\to\infty$, seems to be an interesting topic for future research.

\section{The non-Abelian Toda lattice}

The non-Abelian Toda lattice is a Hamiltonian system which describes the evolution of a system
of particles $X_1,\ldots,X_N$ in the space of invertible $n\times n$ matrices.
There is a standard version, introduced by A.M. Polyakov~\cite{bmrl80},
which generalises the classical Toda lattice.
There is also an {\em indefinite} version, in which the potential has the
opposite sign, as considered by Popowicz~\cite{pop81,pop83}.  
As in the scalar case~\cite{noc13,noc15}, 
it is the indefinite version which is relevant to our setting.

Writing $A_i=X_{i+1} X_i^{-1} $ and $B_i=\dot{X}_i X_i^{-1}$, the Hamiltonian
is given by
$$H = \tr\left( \frac12\sum_{i=1}^N B_i^2 - \sum_{i=1}^{N-1} A_i\right) ,$$
and the equations of motion are
\be\label{emab}
\dot{B}_i = A_{i-1}-A_i ,\qquad \dot{A}_i=B_{i+1} A_i - A_i B_i  .
\ee
This system admits the Lax representation $\dot L=[L,M]$, where
\be
L=\begin{pmatrix} B_1 & -1 & 0 &  \cdots & 0\\
A_1 & B_2 & -1 &  \cdots & 0\\
0 & A_2 & \smash{\ddots} &  \smash{\ddots} &  \smash{\vdots}\\
\smash{\vdots}&&&B_{N-1}&-1\\
0&\dots&0&A_{N-1}&B_N \end{pmatrix},
\qquad
M=\begin{pmatrix} 0 & 0    &\dots& 0\\
A_1 & 0   &\dots& 0\\
\vdots   & \ddots & \ddots & \vdots\\
0&\dots&A_{N-1}&0 \end{pmatrix}.
\ee
Here we are using the notation $[L,M]=LM-ML$, where multiplication
of matrices with matrix-valued entries is ordered, that is,
$$(LM)_{ij}=\sum_k L_{ik}M_{kj}.$$
The Lax representation implies that, for each positive integer $k$,
$C_k=\sum_i (L^k)_{ii}$ is a constant of motion for the system, that is $\dot{C}_k=0$.
Note that these constants of motion are matrix-valued.
In the following we will use the notation $L=L^{(N)}(A,B)$ for the above Lax matrix,
and $C^{(N)}_k$ for the corresponding constants of motion.

The equations of motion \eqref{emab} can be written, equivalently, as
\be\label{emxp}
\dot{X}_i=P_i,\qquad \dot{P}_i=P_iX_i^{-1}P_i+X_iX_{i-1}^{-1}X_i-X_{i+1}.
\ee
In the scalar case, writing $X_i=e^{x_i}$, these reduce to the (indefinite) Toda equations
$$\ddot{x}_i = e^{x_i-x_{i-1}}-e^{x_{i+1}-x_i}.$$
Observe that the space $\P^N\times \S^N$ is invariant under the evolution \eqref{emxp}, 
where $\S$ denotes the set of real symmetric $n\times n$ matrices.
The system therefore admits a natural quantisation in this space, with Hamiltonian
given by \eqref{hnct}.

The diffusion with generator $G_\l$, defined by \eqref{G} with $\l\in\R^N$, is in fact a stochastic 
version of a series of B\"acklund transformations between non-Abelian Toda systems with 
different numbers of particles, as we shall now explain.  

Recall the kernel function $Q^{(N)}_{\nu}(X,Y)$ defined by \eqref{kf},
and consider the following evolution in $\P^N\times\P^{N-1}$:
\be\label{XE}
B_i=\dot{X}_i X_i^{-1} = \vartheta_{X_i} \ln Q^{(N)}_{\nu}(X,Y)=\nu+Y_i X_i^{-1}-X_i Y_{i-1}^{-1},\ 1\le i\le N,
\ee
\be\label{YE}
B_i'=\dot{Y}_i Y_i^{-1} = - \vartheta_{Y_i} \ln Q^{(N)}_{\nu}(X,Y)=\nu +Y_i X_i^{-1}- X_{i+1}Y_i^{-1},\ 1\le i\le N-1.
\ee
Here we are using the conventions $X_{N+1}=Y_N=Y_0^{-1}=0$.
If \eqref{XE} and \eqref{YE} hold, then it is straightforward to compute
\be\label{back}
\dot B_i=A_{i-1}-A_i,\ 1\le i\le N,\qquad \dot B_i'=A'_{i-1}-A'_i,\ 1\le i\le N-1,
\ee
where $A_i=X_{i+1} X_i^{-1} $, $A_i'=Y_{i+1} Y_i^{-1}$, with conventions $A_0=A_N=0$
and $A'_0=A'_{N-1}=0$.
As such, this defines a B\"acklund transformation between the $N$- and $(N-1)$-particle systems.
In the case $\nu=0$, it can be seen as a degeneration of the auto-B\"acklund transformation 
for the semi-infinite non-Abelian Toda lattice described in the paper~\cite{pop83}.
Moreover, the constants of motion are related by $C^{(N)}_k=C^{(N-1)}_k+\nu^k I$.
This follows from the relation, readily verified from~\eqref{XE} and \eqref{YE}:
\be\label{dressing}
L^{(N)}(A,B)\, D(X,Y) = D(X,Y)\, \hat L^{(N-1)}(A',B',\nu),
\ee
where 
$$\hat L^{(N-1)}(A',B',\nu)=\begin{pmatrix} 
&&&0\\
& L^{(N-1)}(A',B')&&  \vdots\\
&&&  0\\
&&&-1\\
0&\cdots&0&\nu I \end{pmatrix},$$
and
$$D(X,Y)=\begin{pmatrix} 1 & 0 &  \cdots & 0&0\\
X_2 Y_1^{-1} & 1 &   \cdots & 0&0\\
0 & X_3 Y_2^{-1} & \smash{\ddots} &    \smash{\vdots}&    \smash{\vdots}\\
\smash{\vdots}&&&1&0\\
0&\dots&0&X_N Y_{N-1}^{-1}&1 \end{pmatrix}.$$

This B\"acklund transformation may be used to construct solutions for the 
$N$-particle system, recursively, as follows.  
Let $\T$, $\T(X)$, $\F(Y)$ and $e_\l(Y)$ be defined as in Section \ref{energy}.
Let $\l\in\R^n$ and set $\F_\l(Y)=\F(Y)-\ln e_\l(Y)$. 
By Proposition~\ref{wc} (ii),  $\F_\l(Y)\to+\infty$ as $Y\to\partial\T(X)$, 
hence there exists $Y^*(\l,X)\in \T(X)$ at which $\F_\l(Y)$ achieves its minimum value on $\T(X)$.
Moreover, this minimiser must satisfy
\be\label{cpe}
\vartheta_{Y^m_i} \F_\l(Y) = 0,\qquad 1\le i\le m <N.
\ee
Equivalently,
$$\l_m + Y^{m-1}_i (Y^m_i)^{-1} - Y^m_i (Y^{m-1}_{i-1})^{-1}
= \l_{m+1} + Y^m_i (Y^{m+1}_i)^{-1} - Y^{m+1}_{i+1} (Y^{m}_{i})^{-1},$$
with the conventions $Y^{m-1}_m=0$ and $(Y^{m-1}_{0})^{-1} =0$.
The equations with $2\le m\le N-1$ are equivalent to
$$\vartheta_{Y^m_i} \ln Q^{(m)}_{\l_m}(Y^m,Y^{m-1})=- \vartheta_{Y^m_i} \ln Q^{(m+1)}_{\l_{m+1}}(Y^{m+1},Y^m).$$
Denote by $\T_\l$ the set of $Y\in\T$ which satisfy the critical point equations \eqref{cpe},
and consider the evolution on $\T$ defined by $B^1_1=\dot Y^1_1 (Y^1_1)^{-1}=\l_1$ and, for $2\le m\le N$,
\be\label{eT}
B^m_i = \dot Y^m_i  (Y^m_i)^{-1} = \vartheta_{Y^m_i} \ln Q^{(m)}_{\l_m}(Y^m,Y^{m-1}),\qquad 1\le i\le m\le N.
\ee
Note that this is equivalent to
\be
B^m_i = \l_m + Y^{m-1}_i (Y^m_i)^{-1} - Y^m_i (Y^{m-1}_{i-1})^{-1} ,\qquad
1\le i\le m\le N,
\ee
with the conventions, as above,  $Y^{m-1}_m=0$ and $(Y^{m-1}_{0})^{-1} =0$.
It corresponds precisely to the drift term of the diffusion in $\T$
with infinitesimal generator $G_\l/2$.  

It follows from \eqref{back} that $\T_\l$ is invariant under the evolution~\eqref{eT},
as in the scalar ($n=1$) case \cite[Proposition 8.1]{noc13}.
Thus, if $Y(0)=Y^*(\l,X)$ and we let $Y(t)$ evolve according to \eqref{eT}, 
then $Y^N(t),\ t\ge 0$ is a realisation of the $N$-particle non-Abelian 
Toda flow on $\P^N$, with $Y^N(0)=X$ and $C^{(N)}_k=\sum_i \l_i^k I$.
In the scalar case, this agrees with the statement of \cite[Theorem 8.4]{noc13}.

To conclude, Theorem~\ref{QNCT} says that if we add noise to the evolution~\eqref{eT}, 
and choose the initial law on $\T$ appropriately, then $Y^N(t),\ t\ge 0$ evolves as a 
diffusion in $\P^N$ with generator given by a Doob transform of the 
quantised Hamiltonian \eqref{hnct}.

\section{A related class of processes}\label{rcp}

Let $G$ be a right-invariant Brownian motion in $GL(n)$, satisfying $\partial G = \partial\beta\, G$,
where $\beta$ is a Brownian motion in $\mathfrak{gl}(n,\R)$ with each entry having infinitesimal variance $1/2$.
Then $Y=G^tG$ is a Brownian motion in $\P$.  As the evolution of $Y$ is governed by $\Delta$, its law is invariant 
under the action of $GL(n)$ on $\P$.

Norris, Rogers and Williams~\cite{nrw} consider the closely related Markov
process $X=GG^t$.  This process has the same eigenvalues as $Y$, but its eigenvectors 
behave quite differently.  It satisfies $\partial X = \partial\beta\, X+X\, \partial\beta^t$, from which the 
Markov property is evident, and one may compute its infinitesimal generator:
$$\Omega_X=\frac12\tr(\vartheta_X\vartheta_X'+\vartheta_X'\vartheta_X)
=\tr(X^2\partial_X^2)+\frac12\tr\vartheta_X + \frac12 \tr(X)\tr\partial_X.$$
Here, $\vartheta_X' f = (\partial_X f)X$.
The differential operator $\Omega$ is easily seen to be $O(n)$-invariant, 
but not $GL(n)$-invariant.
Nevertheless, it bears many similarities to the Laplacian.  For example, it agrees with
the Laplacian when applied to radial functions on $\P$, as can be seen by noting that 
$\vartheta_X' \tr X^k = \vartheta_X \tr X^k$, for all positive integers $k$ (cf. \eqref{txk}).
It is self-adjoint with respect to $\mu$,
and invariant under the change of variables $Y=X^{-1}$.  If $k(X,Y)=\etr(-YX^{-1})$,
then $\Omega_Xk=\Omega_Y k$.  An important difference is the associated 
product rule, cf. \eqref{pr}:
\be\label{opr}
\Omega(fg)=(\Omega f)g+f (\Omega g)+\tr[X^2(\partial_Xf )(\partial_X g)]+\tr[X^2(\partial_Xg)( \partial_X f)].
\ee

Many of the diffusions we have considered have natural analogues in which 
the underlying motion of particles is governed by $\Omega$ rather than $\Delta$.
Let $G$ be a right-invariant Brownian motion in $GL(n)$ with drift $\nu/2$, started at $I$.
Then $Y=G^t G$ is a Brownian motion in $\P$ with drift $\nu/2$, and
$X=GG^t$ is a diffusion in $\P$ with generator 
$$\Omega_X^{(\nu/2)}=\Omega_X+\nu\tr\vartheta_X.$$
We note that this is the Doob transform of $\Omega_X$ via the positive eigenfunction $|X|^{\nu/2}$.

Consider the process in $\P$ defined by $Q = G (Q_0+A) G^{t}$, where $A_t=\int_0^t Y_s^{-1} ds$ 
and $Q_0$ is independent of $G$.  This is a diffusion in $\P$ with infinitesimal generator
$$\cR=\Omega_X^{(\nu/2)}+\tr\partial_Q.$$
This process (with a different normalisation) was studied in \cite{rv16},
where it was observed that $\cR$ is self-adjoint with respect to the measure \eqref{rbm}, as
in the case of a Brownian particle.

For several particles with one-sided interactions, the interactions need to be modified on 
account of the product rule \eqref{opr}.  For example, we may consider
$$\T = \Omega^{(\l)}_Y+\Omega_X +\tr[(XYX^{-1}+X^{-1}YX)\partial_X].$$
Assume $2\l>n-1$.  Then $\Omega^{(\l)} \circ K_\l = K_\l \circ \T$,
where $(K_\l f)(X)= \int_\P f(X,Y) \Pi_\l(X,dY)$ and $\Pi_\l$ is defined by \eqref{pia}.
Assuming the associated martingale problem is well posed, this yields the following
analogue of the `Burke' Theorem~\ref{ln-burke}:
if $(X_t,Y_t)$ be a diffusion in $\P^2$ with generator $\T$
and initial law $\delta_{X_0}(dX) \Pi_{\l}(X,dY)$ then, with respect to its own filtration,  
$X_t$ is a diffusion in $\P$ with generator $\Omega^{(\l)}_X$, started at $X_0$;
moreover, the conditional law of $Y_t$, given $X_s,\ s\le t$, only depends on $X_t$
and is given by $ \Pi_{\l}(X_t,dY)$. 

\section{The complex case}

Most of the discussion in this paper carries over naturally to the complex setting.
We remark in particular that the matrix Dufresne identity and related $2M-X$ theorem are studied
in some detail in the complex setting by Rider and Valk\'o~\cite{rv16} and Bougerol~\cite{bougerol}.
In this section, we briefly outline how the framework developed in this paper may be extended to the 
complex setting, with particular emphasis on a complex version of the intertwining relation \eqref{int-duf} 
and some of its consequences.  We will also briefly indicate how this relates to a remarkable identity
of Fitzgerald and Warren~\cite{fw} concerning Brownian motion in a Weyl chamber,
and closely related work of Nguyen and Remenik~\cite{nr17}
and Liechty, Nguyen and Remenik~\cite{lnr}
on non-intersecting Brownian bridges.

In this section, $\P$ will denote the space of positive $n\times n$ Hermitian matrices.  
For $a\in GL(n,\C)$ and $X\in\P$, write $X[a]=a^\dagger X a$.  
This defines an action of $GL(n,\C)$ on $\P$.
The $GL(n,\C)$-invariant volume element on $\P$ is given by $\mu(dX)=|X|^{-n} dX$, 
where $dX$ denotes the Lebesgue measure on $\P$.
We note that, writing $A=a[k]$, where $k\in U(n)$ and $a$ is the diagonal matrix with entries given by the eigenvalues
$z_1,\ldots,z_n$ of $A$, with $z\in C_+=\{z_1>\cdots>z_n>0\}$, we have the decomposition, on $U(n)\times\R_+^n$,
\be\label{mu}
\mu(dA) = \pi^{n(n-1)/2} \left(\prod_{j=1}^{n-1} j!\right)^{-1} dk \prod_{i<j}(z_i-z_j)^2 \prod_i  z_i^{-n} dz_i.
\ee

The Laplacian on $\P$ may be characterised (see, for example, \cite{nomura}) as
the unique invariant differential operator on $\P$ which satisfies
\be\label{laplace}
\Delta_X \etr(X) =  (n\tr X + \tr X^2) \etr(X).
\ee
Let us define a `partial matrix derivative' $\partial_X$ on $\P$, writing $X=(x_{ij})$, by
$$(\partial_X)_{ij} = \begin{cases} \frac{\partial}{\partial x_{ii}} & i=j\\
\frac{\partial}{\partial \bar{x}_{ij}} & i\ne j\end{cases}$$
where
$$\frac{\partial}{\partial \bar{z}} = 
\frac12\left(\frac{\partial}{\partial \Re z}+\sqrt{-1} \frac{\partial}{\partial \Im z}\right) .$$
If $Y=X[a]$ for some fixed $a\in GL(n)$, then 
$ \partial_X=a\, \partial_Y \, a^\dagger$.
This implies that $\tr \vartheta_X^2$ is $GL(n,\C)$-invariant, where $\vartheta_X = X \partial_X$.
One can check that $ \partial_X \tr(X)=I$ and $\partial_X X=n I$.  This implies
that $\tr \vartheta_X^2$ satisfies \eqref{laplace} and hence
$\Delta_X = \tr \vartheta_X^2$, as in the real case.

In this setting, the spherical functions are defined as follows.  
For $\l,x\in\C^n$, define $a_\l(x)=\det(x_i^{\l_j})$ and
$s_\l(x)=a_{\delta+\l}(x)/a_\delta(x)$, where $\delta_i=n-i$.
For $X\in\P$,
define $s_\l(X)$ to be the function $s_\l$ evaluated at the eigenvalues of $X$.

We define Brownian motion in $\P$ to be the diffusion process
with generator $\Delta/2$.  If $\varphi$ is a positive eigenfunction
of the Laplacian with eigenvalue $\gamma$, we may consider the 
corresponding Doob transform $\Delta_X^{(\varphi)}/2$, and call the diffusion 
process with this generator a Brownian motion in $\P$ with drift $\varphi$.
In particular, for $\l\in\R^n$, we may consider the positive eigenfunction $s_\l$,
where $\l$ is such that $\rho+\l\in \bar C$, where $\rho_i=(n-2i+1)/2$
and $C=\{x\in\R^n:\ x_1>\cdots>x_n\}$.
Then, as is well known, the logarithmic eigenvalues of a Brownian motion on $\P$ 
with drift $s_\l$ evolve as a standard Brownian motion in $\R^n$ with drift $\rho+\l$ 
conditioned never to exit $C$.

With these ingredients in place, all of the basic calculus and intertwining relations 
discussed previously carry over, with the obvious modifications.  We will briefly 
illustrate this here in the context of the intertwining relation \eqref{int-duf} and
consider some of its consequences.  

\subsection{Matrix Dufresne identity}

As in the real case, it is known \cite{mj,bfj} that for each $X,Y\in\P$, the equation 
$Y=AXA$ has a unique solution in $\P$, namely
$$A=X^{-1/2}(X^{1/2}YX^{1/2})^{1/2} X^{-1/2}.$$
Moreover, if $X$ is fixed, then
\be\label{ccr2}
\partial_A=XA\, \partial_Y+\partial_Y AX.
\ee
Let $M=\Delta_Y +\tr(Y\partial_A)$.  If $Y=AXA$ then, in the variables $(X,A)$, we can write
$$M=\Delta_X -2\tr(XAX\partial_X)+\tr(AXA\partial_A).$$
This follows from \eqref{ccr2}, as in the real case.

Let us define
$$J=\Delta_X-\tr X,\qquad
p(X,A)=\etr(-AX-A^{-1})$$
and the corresponding integral operator
$$(Pf)(X)=\int_\P p(X,A) f(X,A) \mu(dA).$$
Then, on a suitable domain, the following intertwining relation holds:
\be\label{cint-my}
J \circ P = P\circ M.
\ee
The proof is identical to the real case.

Note that the intertwining relation \eqref{cint-my} implies
\be\label{cint-duf}
J \circ D = D \circ \Delta ,
\ee
where $D$ is the linear operator defined, for suitable $f:\P\to\C$ by
$$(Df)(X)=\int_\P f(AXA) \etr(-AX-A^{-1}) \mu(dA).$$

Now suppose $\varphi$ is a positive eigenfunction of $\Delta$ with eigenvalue $\gamma$
such that $\beta=D\varphi<\infty$.
Then it follows from \eqref{cint-duf} that $\beta$ is a positive eigenfunction of $J$ with eigenvalue $\gamma$.
Note that if we write $\beta(X)=\varphi(X) B_\varphi(X)$,
then this implies
\be\label{cFK}
(\Delta_X^{(\varphi)}-\tr X)B_\varphi(X)=0.
\ee
As the real case, for suitable $\varphi$, the function $B_\varphi$ 
admits a natural probabilistic interpretation, via the Feynman-Kac formula.
\begin{prop}\label{cfk-my} Let $\varphi$ be a positive eigenfunction of $\Delta$
such that $D\varphi<\infty$, and the martingale problem associated with
$\Delta^{(\varphi)}$ is well posed for any initial condition in $\P$.
Let $Y$ be a Brownian motion in $\P$ with drift $\varphi$ started at $X$, and 
denote by $\E_X$ the corresponding expectation.  Assume that, for any $X\in\P$, 
\be\label{cfinite}
Z=\frac12 \int_0^\infty \tr Y_s\ ds<\infty
\ee
almost surely, and define $M_\varphi(X)=\E_X e^{-Z}$.
Suppose also that $\lim_{X\to 0} M_\varphi(X) =1$ and 
\be\label{cBC}
\qquad \lim_{X\to 0} B_\varphi(X) = C_\varphi ,
\ee
where $C_\varphi>0$ is a constant.  Then $B_\varphi(X)=C_\varphi\ M_\varphi(X)$
and, moreover, $B_\varphi$ is the unique bounded solution to \eqref{cFK} satisfying 
the boundary condition \eqref{cBC}.
\end{prop}
Again the proof is identical to the real case.
We obtain the following corollary, in agreement with \cite[Corollary 10]{rv16}.
\begin{cor}
If $\varphi(X)=|X|^{-\nu/2}$, where $\nu>n-1$, then $Z$
 is inverse complex Wishart distributed with density
proportional to $$|A|^{-\nu} \etr(-A^{-1}) \mu(dA).$$
\end{cor}

Let $\l\in \R^n$.
It follows from the above that $\beta_\l:=Ds_\l$, which is easily seen to be finite, is a eigenfunction of $J$. 
Thus, setting $B_\l(X)=s_\l(X)^{-1}\beta_\l(X)$, we have
\be\label{FKB}
(\Delta_X^{(s_\l)}-\tr X) B_\l(X)=0.
\ee

Let us now assume that $\l\in C$, with $2\l_1<1-n$.  Then
$$c_\l = s_\l(I)^{-1} \int_\P s_\l(A^2) \etr(-A^{-1}) \mu(dA)<\infty,$$
and, by uniqueness of the spherical functions,
$$\int_\P s_\l(AXA) \etr(-A^{-1}) \mu(dA)=c_\l s_\l(X).$$
Recall the representation \eqref{mu} and note that, writing $\mu=-\l-\rho$,
$$s_{\l}(A^2) = a_{-\mu}(A^2)/a_\rho(A^2) = \det\left(z_i^{-2\mu_j} \right)
\prod_{i<j}(z_i^2-z_j^2)^{-1} \prod_i  z_i^{n-1},$$
where $z_1,\ldots,z_n$ denote the eigenvalues of $A$.
Thus, using Schur's Pfaffian identity and de Bruijn's formula~\cite{deb}, 
we may compute
\be\label{cl}
c_\l =   \prod_i \Gamma(2\mu_i) \, \prod_{i<j}\frac{\pi}{\mu_i+\mu_j}.
\ee

Let $Y$ be a Brownian motion in $\P$ with drift $s_\l$, started at $X$,
and denote by $\E_X$ the corresponding expectation.
The condition $2\l_1<1-n$ ensures that $Z<\infty$, almost surely.
As in the real case, using the fact that 
the law of $Y$ is invariant under multiplication by scalars,
it is easy to see that $\E_X e^{-Z} \to 1$ and $B_\l(X)\to c_\l$
as $X\to 0$ along any ray in $\P$.  To prove that these limits
exist as $X\to 0$ in $\P$ is somewhat technical, and we will 
not pursue it here, but rather state it as a hypothesis:
\begin{hyp}\label{hyp}
As $X\to 0$ in $\P$, $\E_X e^{-Z} \to 1$ and $B_\l(X)\to c_\l$.
\end{hyp}

Assuming Hypothesis \ref{hyp}, it follows from Proposition \ref{cfk-my} that
\be\label{Bmgf} B_\l(X) = c_\l\ \E e^{-Z}.\ee
Again using the fact that the law of $Y$ is invariant under multiplication by scalars,
the identity \eqref{Bmgf} is equivalent to the statement that $Z$ 
has the same law as $\tr(AX)$,
where $A$ is distributed according to the probability measure
$$c_\l^{-1} \, s_\l(X)^{-1} \, s_\l(AXA)\, \etr(-A^{-1})\, \mu(dA).$$
As this is a statement about the eigenvalues of $Y$, we may
rephrase it in terms of a Brownian motion in $C$, as follows.
Let $-\mu,x\in C$ and assume that $\mu_1>0$.
Let $\xi(t)$ be a Brownian motion with drift $-\mu$, 
started at $x$ and conditioned never to exit $C$.
Let 
$$Z=\frac12 \sum_{i=1}^n \int_0^\infty e^{\xi_i(s)} ds.$$
Then, assuming Hypothesis \ref{hyp}, $Z$ has the same law as $\tr(AX)$,
where $X$ has eigenvalues $e^{x_1},\ldots,e^{x_n}$ and $A$ is distributed according to
$$c_{-\mu-\rho}^{-1}\, s_{-\mu-\rho}(X)^{-1} \, s_{-\mu-\rho}(AXA)\, \etr(-A^{-1})\, \mu(dA) .$$
If $\xi(t)$ is started at the origin, then $Z$ has the same law as 
$\sum_i z_i$, where $z$ has density
\be\label{gb} b_\mu^{-1} \det\left(z_i^{-2\mu_j} \right)\, \prod_{i<j}\frac{z_i-z_j}{z_i+z_j} \,\prod_i  e^{-1/z_i} \frac{dz_i}{z_i} ,\ee
on $C_+$.  The normalisation constant $b_\mu$ is given by
\be\label{bmu}
b_\mu = \prod_i \Gamma(2\mu_i) \, \prod_{i<j} \frac{\mu_j-\mu_i}{\mu_i+\mu_j}.
\ee

The measure \eqref{gb} is a generalised Bures measure and defines a Pfaffian point process related to the 
BKP hierarchy~\cite[Proposition 4.4]{wl}.  

To conclude, consider the function 
$$\tilde B_\l(X) =  \prod_{i<j}\frac{\pi}{\mu_i+\mu_j} \,
\det(e^{-\mu_i x_j})^{-1}\, \det\left( 2 K_{2\mu_i}(2e^{x_j/2})\right),$$
where $K_\nu$ is the Macdonald function,
and $e^{x_1},\ldots,e^{x_n}$ are the eigenvalues of $X$.
It can be shown that $\tilde B_\l(X)$ satisfies \eqref{FKB}
and $\lim_{X\to 0}B_\l(X)=c_\l$, so if Hypothesis \ref{hyp} holds then $B_\l(X) =\tilde B_\l(X)$.

\subsection{Maximum of Dyson Brownian motion with negative drift}

Let us write $2\mu=\epsilon \alpha$, where $-\alpha\in C$, $\alpha_1>0$ and $\epsilon>0$.
Changing variables to $y_i=\epsilon \ln z_i$, the density \eqref{gb} becomes,
up to a constant factor,
$$ \det(e^{-\alpha_i y_j}) \prod_{i<j}\tanh\left(\frac{y_i-y_j}{2\epsilon}\right) \prod_i  \exp(-e^{-y_i/\epsilon}) dy_i .$$
In the limit as $\epsilon\to 0$, this reduces to
\be\label{loe}
\det(e^{-\alpha_i y_j})  \prod_i  dy_i ,
\ee
on the set $C_+=\{y_1>\cdots>y_n>0\}$.  Note that, in this scaling limit, 
$$ \epsilon \ln \sum_i e^{y_i/\epsilon} \sim y_1.$$
If $\alpha_i=\alpha$, for all $i$, the measure \eqref{loe} reduces to a special
case of the Laguerre Orthogonal Ensemble (LOE) of random matrix theory.

On the other hand, by Brownian scaling and the method of Laplace, as $\epsilon\to 0$ the random
variable $\epsilon \ln Z$ converges in law to $\sup_t \eta_1(t)$, where $\eta(t)$ is a Brownian motion 
in $\R^n$ with drift $-\alpha/2$, conditioned to stay in $C$, and started at the origin.  Putting these
observations together we conclude (modulo technicalities) that $\sup_t \eta_1(t)$ has the same law 
as the first coordinate of the ensemble in $C_+$ with density given by \eqref{loe},
 in agreement with \cite[Theorem 1]{fw}.

We note that one may compute the law of $\sup_t \eta_1(t)$ directly, 
for a general initial condition, as follows.
Consider the process $\xi(t)$, a BM with drift $\nu\in C_+$,
started at $y\in C_+$ and conditioned never to exit $C$.
Let $T$ be the first exit time of $\xi$ from $C_+$.  Then, using formulas from \cite{BBO05}
for the exit probabilities of a Brownian motion with drift from $C$ and $C_+$,
$$f(\nu,y):=P(T=\infty)=\det(e^{\nu_i y_j}-e^{-\nu_i y_j})/\det(e^{\nu_i y_j}).$$
This implies that, if $\eta(t)$ is a Brownian motion 
in $\R^n$ with drift $-\alpha/2$, conditioned to stay in $C$, and started at $x\in C$,
and $M=\sup_t \eta_1(t)$, then, for $\alpha_1>0$ and $z\ge x_1$, we have the
formula $P(M\le z)=f(\nu,y)$, where $\nu_i=\alpha_{n-i+1}/2$ and $y_i=z-x_{n-i+1}$.

We remark that $f(v,y)$ may also be interpreted as the probability that a collection
of non-intersecting Brownian bridges of unit length, started at positions $\nu_1,\ldots,\nu_n$
and ending at positions $y_1,\ldots,y_n$, do not exit the domain $C_+$.  
This follows from the reflection principles associated with $C$ and $C_+$, and
gives a formula for the distribution function of the maximum of the `top' bridge,
namely $F(r)=f(\nu',y')$, where $\nu_i'=r-\nu_i$ and $y_i'=r-y_i$.
In this context,
the connection to the LOE was first discovered by Nguyen and Remenik~\cite{nr17}
for bridges starting and ending at the origin, and for general starting and ending
positions by Liechty, Nguyen and Remenik~\cite{lnr}, where further
asymptotic results are obtained in terms of Painlev\'e II and the KdV equation.

\appendix

\section{Markov functions}\label{mfa}

The theory of Markov functions is concerned with the question: 
when does a function of a Markov process 
inherit the Markov property?  
The simplest case is when there is symmetry in the problem, for example, the
norm of Brownian motion in $\R^n$ has the Markov property, for any initial condition, 
because the Laplacian on $\R^n$ is invariant under the action of $O(n)$.  
A more general formulation of this idea is the well-known {\em Dynkin 
criterion}~\cite{dynkin}.  
There is another, more subtle, criterion which has been proved at various levels 
of generality by, for example, Kemeny and Snell~\cite{ks1}, Rogers and Pitman~\cite{rp} and Kurtz~\cite{kurtz}.  
It can be interpreted as a time-reversal of Dynkin's criterion~\cite{kelly} and provides sufficient conditions
for a function of a Markov process to have the Markov property, but only for very particular initial conditions.
For our purposes, the martingale problem formulation of Kurtz~\cite{kurtz} is best suited, as it is quite flexible 
and formulated in terms of infinitesimal generators.

Let $E$ be a complete, separable metric space. 
Denote by $B(E)$ the set of Borel measurable functions on $E$, by $C_b(E)$ the set of bounded 
continuous functions on $E$ and by $\P(E)$ the set of Borel probability measures on $E$.  
Let $A:\D(A)\subset B(E)\to B(E)$ and $\nu\in \P(E)$.
A progressively measurable $E$-valued process $X=(X_t,\ t\ge 0)$ is a solution to the {\em martingale
problem} for $(A,\nu)$ if $X_0$ is distributed according to $\nu$ and there exists a filtration $\F_t$ such that
$$f(X_t)-\int_0^t Af(X_s) ds$$
is a $\F_t$-martingale, for all $f\in\D(A)$.  The martingale problem for $(A,\nu)$ is {\em well-posed} is there exists a
solution $X$ which is unique in the sense that any two solutions have the same finite-dimensional distributions.

The following is a special case of Corollary 3.5 (see also Theorems 2.6, 2.9 and the remark at the top 
of page 5) in the paper \cite{kurtz}.

\begin{thm}[Kurtz, 1998]\label{mf} Assume that $E$ is locally compact, that $A:\D(A)\subset C_b(E) \to C_b(E)$, 
and that $\D(A)$ is closed under multiplication, separates points and is convergence determining.
Let $F$ be another complete, separable metric space, $\gamma:E\to F$ continuous and
$\L (x,dz)$ a Markov transition kernel from $F$ to $E$ such that 
$\L (g\circ\gamma)=g$ for all $g\in B(F)$,
where $\L f(x)=\int_E f(z) \L (x,dz)$ for $f\in B(E)$.
Let $B:\D(B)\subset B(F)\to B(F)$, where $\L (\D(A))\subset\D(B)$, and suppose
$$B\L f=\L Af,\qquad f\in\D(A).$$
Let $\mu\in\P(F)$ and set $\nu=\int_F \mu(dx) \L (x,dz)\in\P(E)$.  
Suppose that the martingale problems
for $(A,\nu)$ and $(B,\mu)$ are well-posed, and that $Z$ is a solution to the martingale problem for
$(A,\nu)$.  Then $X=\gamma\circ Z$ is a Markov process and a solution to the martingale problem for 
$(B,\mu)$.  Furthermore, for each $t\ge 0$ and $f\in B(E)$ we have, almost surely, 
$$\E[f(Z_t)|\ X_s,\ 0\le s\le t]=\int_E f(z) \L (X_t,dz).$$
\end{thm}

To apply this theorem to the examples considered in Sections 1--9, 
we take $E=\P^a$, $F=\P^b$, where $a>b$, and $\gamma(X,Y)=X$.
In all of our examples, we have
$$A=\sum_{i=1}^{a} [\Delta_{Z_i} + \tr(a_i(Z)\partial_{Z_i})],
\qquad B=\sum_{i=1}^{b} [\Delta_{X_i} + \tr(b_i(X)\partial_{X_i})],$$
where the $a_i$ and $b_i$ are locally Lipschitz functions on $\P^a$
and $\P^b$, respectively.  
For such generators, we may take $\D(A)=C^2_c(\P^a)$
and $\D(B)=C^2_c(\P^b)$.
In our examples, the intertwining operator always has the form
$$\L f(X) = \int_{\P^{a-b}} k(X,Y) f(X,Y) \mu_{a-b}(dY),$$
where $k\in C^2(\P^a)$.
This ensures that $\L (g\circ\gamma)=g$ for all $g\in B(\P^b)$,
and $\L (\D(A))\subset\D(B)$.
Thus, in all of our examples, as long as the martingale problems associated
with $A$ and $B$ are well posed, for arbitrary initial conditions, and
the intertwining relation $B\L f=\L Af$ holds for $f\in C^2_c(\P^a)$,
then the conclusions of Theorem \ref{mf} are valid.

\section{Well-posedness of martingale problems}\label{wpmp}

In the following, 
$\P$ is the space of positive $n\times n$ real symmetric matrices.
In all of the examples we consider (in Sections 1--9)
we have generators of the form
$$L=\sum_{i=1}^r [\Delta_{X_i} + \tr(b_i(X)\partial_{X_i})],$$
where the $b_i$ are locally Lipschitz functions on $\P^r$.
The corresponding SDE therefore admits
a unique strong solution, for any given initial condition, up to an explosion time $\tau$.
If $\tau=\infty$ almost surely, then the corresponding martingale problem is well-posed~\cite{ek,k11}.

To show that $\tau=\infty$ almost surely, it suffices to find a positive function $U$ and
constants $c,d$ such that $U(X)\to\infty$ as $X\to\partial\P^r$ and $LU\le cU+d$.
Such a function $U$ is called a {\em Lyapunov function}.
In the following, we exhibit Lyapunov functions for some of the generators 
considered in this paper.

\bigskip
\noindent{\bf Example 1}\ \ The function \be\label{C} C(X)=\tr(X)+\tr(X^{-1})\ee
is a Lyapunov function for the Laplacian on $\P$.  
Indeed, using \eqref{d2ax} we have
$$\Delta C=\frac{n+1}2 C.$$
Another choice is 
\be\label{D} D(X)=\tr(X)-\ln |X| .\ee
This is positive, since $e^x>x$.
Using \eqref{d2ax}, \eqref{de} and the inequality $e^x>2x$, we have
$$\Delta D = \frac{n+1}2 \tr(X) \le (n+1) D,$$
as required.

The functions $C$ and $D$ are also Lyapunov functions for 
$\Delta^{(\nu)}_X=\Delta_X+2\nu\,\tr\vartheta_X$, since
$$\Delta^{(\nu)} C \le \left( \frac{n+1}2 +2|\nu|\right) C,\qquad
\Delta^{(\nu)} D \le (n+1+4\nu^+) D +2\nu^- n.$$

More generally, suppose $\varphi$ is a positive eigenfunction of $\Delta$
with eigenvalue $\l$, and
$$\Delta_X^{(\varphi)} = \varphi(X)^{-1} \circ (\Delta_X-\l) \circ \varphi(X) =
\Delta_X+2\tr(\vartheta_X\ln\varphi(X) \ \vartheta_X).$$
For positive integers $k,l$, define 
$$ U_{k,l}(X)=\tr X^k + \tr X^{-l},\qquad U^{(\varphi)}_{k,l}(X)=U_{k,l}(X)/\varphi(X).$$
By Lemma~\ref{lyapunov}, 
there exist constants $c_{k,l}$ such that $\Delta U_{k,l} \le c_{k,l} U_{k,l}$,
and hence 
$$\Delta_X^{(\varphi)} U^{(\varphi)}_{k,l} \le (c_{k,l}-\l) U^{(\varphi)}_{k,l} .$$
Thus, if there exist positive integers $k,l$ such that $U^{(\varphi)}_{k,l} (X) \to+\infty$ 
as $X\to\partial\P$, then $U^{(\varphi)}_{k,l}$ is a Lyapunov function for $\Delta_X^{(\varphi)}$.
For example, if $\varphi=h_s$, where $s\in\R^n$, then we may find $p,q>0$
such that 
\be\label{hsb}
h_s(X)\le n^n \tr X^p \tr X^{-q},
\ee
 and hence $U^{(\varphi)}_{k,l}$ is a 
Lyapunov function for $\Delta_X^{(\varphi)}$ provided $k>p$ and $l>q$.
Indeed, recall that 
$$h_s(X)=\int_{O(n)} p_s(X[k])\; dk,\qquad p_s(X)=\prod_{i=1}^n |X^{(i)}|^{s_i},$$
where $X^{(i)}$ denotes the $i\times i$ upper left hand corner of $X$.
Now, for each $i$, 
$$|X^{(i)}|^{s_i} \le \tr [(X^{(i)})^{is_i}] \le \tr (X^{is_i}),$$
hence
$$h_s(X) \le \prod_{i=1}^n \tr (X^{is_i}).$$
Considering positive and negative powers separately,
and repeatedly applying the Chebyshev sum inequality,
yields \eqref{hsb} with $p=\sum_i i s_i^+$ and $q=\sum_i is_i^-$.

\bigskip
\noindent{\bf Example 2}\ \ In Section \ref{onesided}, we encounter 
$$T=\Delta_Y + \Delta_X + 2\tr(Y\partial_X).$$
In the following we continue to make use of the functions $C$ and $D$
defined in the previous example.
Let 
\be\label{V}
V(X,Y)=D(X^{-1/2}YX^{-1/2})=\tr(YX^{-1})-\ln|YX^{-1}|.
\ee

We compute
$$\Delta_X V=\Delta_Y V=\frac{n+1}2 \tr(YX^{-1})$$
and
$$\tr(Y\partial_X)V = \tr(YX^{-1}) - \tr(YX^{-1}YX^{-1}).$$
For the first identity we have used \eqref{d2ax} and \eqref{d2e},
and for the second we have used \eqref{dax} and \eqref{de}.
It follows that
$$TV\le (n+3)\tr(YX^{-1}).$$
Now, using $X^{-1/2}YX^{-1/2}\in\P$ and $e^x>2x$, it holds that
\be\label{trd}
\tr(YX^{-1})\le 2V,
\ee
and so $TV\le 2(n+3)V$.
For the Lyapunov function we can now take
$$U(X,Y)=C(Y)+V(X,Y).$$
Recalling that $ \Delta_Y C(Y) =(n+1)C(Y)/2$,
we obtain $TU\le 2(n+3)U$,
as required.  

One may also consider the same process with drifts:
$$T'=\Delta_Y^{(\l)} + \Delta_X^{(\nu)} + 2\tr(Y\partial_X)=T+2\l\,\tr\vartheta_Y+2\nu\,\tr\vartheta_X.$$
In this case, we have $T' U\le cU+d$, where
$d=2(\nu-\l)^+ n$ and $c=2(n+3)+4(|\l|+|\nu|)$.

\bigskip
\noindent{\bf Example 3}\ \ For $$G=\Delta_Y^{(\l)} + \Delta_{X_1}^{(\nu)} +2\tr(Y\partial_{X_1})+ \Delta_{X_2}^{(\nu)}-2\tr(X_2Y^{-1}X_2\partial_{X_2}),$$
we let $U=C(Y)+ V'$, where $C$ is defined by \eqref{C} and
$$V'(X_1,X_2,Y)=V(X_1,Y)+V(Y,X_2)=\tr(YX_1^{-1})+\tr(X_2 Y^{-1})-\ln|X_1^{-1}X_2|,$$
with $V$ defined by \eqref{V}.
Then it holds that 
$$GV'\le (2n+6+4(|\l|+|\nu|)) V',$$
and hence
\be\label{LG}
GU\le (2n+6+4(|\l|+|\nu|)) U.
\ee

\bigskip
\noindent{\bf Example 4}\ \ Consider
$$L=\psi_{\l,\nu}^{(2)}(X)^{-1} H_{\l,\nu} \psi_{\l,\nu}^{(2)}(X),$$
where
$$H_{\l,\nu}=\Delta_{X_1}+\Delta_{X_2}-2\tr(X_1^{-1}X_2)-n\l^2-n\nu^2.$$
Let $G$ and $U$ be as in the previous example.
From \eqref{int-nct}, we have the intertwining 
$$H_{\l,\nu} \circ Q^{(2)}_{\l,\nu} = Q^{(2)}_{\l,\nu} \circ G.$$
Together with \eqref{LG}, this implies that
$$L\tilde U\le (2n+6+4(|\l|+|\nu|)) \tilde U,$$
where $$\tilde U(X)=\psi_{\l,\nu}^{(2)}(X)^{-1} Q^{(2)}_{\l,\nu} U(X).$$
By Lemma~\ref{mi}, 
$$2U(X,Y)\le  2C(Y) +\tr X_1^{-2} + \tr X_2^2 + \tr Y^2 + \tr Y^{-2}-2\ln |X_1^{-1}X_2|.$$
Together with Lemma~\ref{lem-conv}, this implies that $\tilde U(X)<\infty$ for all $X\in\P^2$.
Finally, by Fatou's lemma, $\tilde U(X)\to\infty$ as $X\to\partial\P^2$, as required.

\bigskip
\noindent{\bf Example 5}\ \ In Section \ref{my}, especially Theorem \ref{mmy}, we encounter
$$M_\nu=\Delta^{(\nu/2)}_Y+\tr(Y\partial_A),\qquad L_\nu=\Delta_X+2\tr(\vartheta_X\ln \beta_\nu(X)\ \vartheta_X).$$
Recall that $L_\nu= J^{(\beta_\nu)}=\beta_\nu(X)^{-1}\circ (J-\l) \circ \beta_\nu(X)$,
where $J=\Delta_X-\tr X$, $\beta_\nu(X)=|X|^{\nu/2} B_\nu(X)$ and $\l=n\nu^2/4$.

For the first process, we can take
$$U(Y,A)=\tr Y+\tr Y^{-1}+\tr A+\tr A^{-1}.$$
Then it holds that $M_\nu U \le c U$, where $c=(n+|\nu|+3)/2$.
It follows, using the intertwining relation \eqref{int-my2}
together with Fatou's lemma (as in the previous example), that $\tilde U= P_{\nu} U$
is a Lyapunov function for the generator $L_\nu$, where
$$(P_{\nu} f)(X) = B_\nu(X)^{-1} \int_\P |A|^\nu \etr(-AX-A^{-1}) f(X,A) \mu(dA).$$
Note that, by Lemma \ref{mi},
$$2U(AXA,A)\le \tr A^4 + \tr X^2 + \tr A^{-4} + \tr X^{-2} + \tr A + \tr A^{-1},$$ 
which together with Lemma \ref{lem-conv} implies 
$$\tilde U(X) = B_\nu(X)^{-1} \int_\P |A|^\nu \etr(-AX-A^{-1}) U(AXA,A) \mu(dA) <\infty.$$
When applying Fatou's lemma, we note that, 
for fixed $A\in\P$, $X\to\partial\P \iff AXA\to\partial\P$.    

\bigskip
\noindent{\bf Example 6}\ \ Here 
we consider the generators $G_\l$ and $L_\l$ defined by \eqref{G} and \eqref{L}, 
respectively, with $\l\in\R^N$.  We follow the same approach as in Examples 3 and 4 above, which treat
the case $N=2$.  If $U$ is a Lyapunov function for $G_\l$ with $\Sigma_\l U<\infty$, then, by the
intertwining relation \eqref{int-hg}, $\tilde U(X)=\psi_\l(X)^{-1} \Sigma_\l U (X)$ is a Lyapunov function for $L_\l$.

Consider the directed graph with vertices $\{(i,m):\ 1\le i\le m\le N\}$, and edges 
$$(i,m) \to (i,m+1),\ 1\le i\le m<N,\qquad (i,m) \to (i-1,m-1),\ 1< i\le m\le N.$$
Let us simplify notation and write, for $a=(i,m)$, $Y_a=Y^m_i$, $\partial_a=\partial_{Y_a}$,
$\vartheta_a = \vartheta_{Y_a}$, $\Delta_a=\Delta_{Y_a}$, etc.
We will also write $\nu_a=\l_m$, for $a=(i,m)$.
Finally, for $a=(i,m)$ denote $a'=(i,m-1)$ and $a''=(i-1,m-1)$, whenever these neighbouring vertices exist.
Then we can write
$$G_\l = \sum_a [\Delta^{(\nu_a)}_a + \delta_a],\qquad \delta_a = 2\tr(Y_{a'}\partial_a)-2\tr(Y_a Y_{a''}^{-1} \vartheta_a).$$
Here we adopt the convention that the term involving $a'$ is defined to be zero when $a=(m,m)$
and similarly the term involving $a''$ is defined to be zero when $a=(1,m)$.  

Let
$$U(Y)=C(Y^1_1) + \sum_{a<b} U_{ab},$$
where 
$$U_{ab}=\tr(Y_a Y_b^{-1})-\ln|Y_a Y_b^{-1}|,$$
and the sum is over all pairs of distinct vertices $a,b$ such that
there is a directed path starting at $a$ and ending at $b$.
We first note that
$$\Delta_{Y^1_1} C(|Y^1_1) = \frac{n+1}2 C(|Y^1_1),$$
and for each pair $a<b$, using \eqref{d2ax} and \eqref{trd},
$$\Delta_a U_{ab} = \Delta_b U_{ab} = \frac{n+1}2 \tr(Y_a Y_b^{-1}) \le (n+1) U_{ab}.$$
Hence
$$\sum_a \Delta_a U \le 2(n+1) U.$$
We also have, using \eqref{dax} and \eqref{de}, respectively,
\bea
(2\nu_a\tr\vartheta_a+2\nu_b\tr\vartheta_b) U_{ab} &=& 2(\nu_a-\nu_b) \tr(Y_a Y_b^{-1}) + 2 (\nu_b-\nu_a) n\\
&\le& 2(\nu_a-\nu_b)^+ U_{ab} + 2 (\nu_b-\nu_a)^+ n,
\eea
and
$$2\l_1 \tr\vartheta_{Y^1_1} C(Y^1_1) \le 2 |\l_1| C(Y^1_1).$$
Combining these gives
$$\sum_a \Delta^{(\nu_a)}_a U \le c U + d,$$
where $c=2(n+1)+2|\l_1|+4\max_{a<b}(\nu_a-\nu_b)^+$ and $d=2n\sum_{a<b}(\nu_b-\nu_a)^+$.

Finally, we note that
$$\delta_a U_{ab} \le 2\tr (Y_{a'} Y_b^{-1} ) + 2 \tr 2\tr (Y_{a} Y_{a''}^{-1} )\le 4 (U_{a'b}+U_{a a''}),$$
and
$$\delta_b U_{ab} \le 2\tr (Y_{a} Y_{b''}^{-1} ) + 2 \tr 2\tr (Y_{b'} Y_{b}^{-1} )\le 4 (U_{ab''}+U_{b' b}),$$
hence
$$\sum_a \delta_a U \le 4 \sum_{a<b} (U_{a'b}+U_{a a''}+U_{ab''}+U_{b' b})\le 2(N+2)(N-1)U.$$
We thus obtain a Lyapunov function $U$ for $G_\l$ with $G_\l U \le c' U+d$ and
$c'=c+2(N+2)(N-1)$.  It remains to show that $\Sigma_\l U' <\infty$.  This is straightforward,
using $2\tr(Y_{a} Y_b^{-1} ) \le \tr Y_a^2+\tr Y_b^{-2}$, and proceeding as in the proof of Proposition \ref{wc}.

\section{Some additional lemmas and proofs}

\subsection{Matrix inequalities}

\begin{lem}\label{mi}
For $A,B,C\in\P$, we have
$$\tr A^2B + \tr C^2 B^{-1} \ge 2 \tr AC.$$
\end{lem}
\begin{proof}
Let $X=B^{1/2}A-B^{-1/2} C$.  Then
$$X^t X=ABA+CB^{-1}C-CA-AC,$$
and $\tr X^tX \ge 0$ implies the result.
\end{proof}

Note that this implies, for $A,B,C\in\P$,
$$\tr AB^{-1}+\tr BC^{-1} \ge 2\tr(A^{1/2}C^{-1/2}).$$
Also, taking $B=C$, this becomes
$$\tr AB^{-1} \ge 2\tr(A^{1/2}B^{-1/2})-n,$$
and we note that iterating this gives
$$\tr AB^{-1} \ge 2^k\tr(A^{1/2^k}B^{-1/2^k})-n(2^k-1).$$

\begin{lem}\label{a2}
For $A_1,\ldots,A_m\in\P$, there exist constants $\alpha,c>0$ and $d\ge 0$
(depending only on $m$)
such that
$$\tr A_1 A_2^{-1}+\tr A_2 A_3^{-1}+\cdots+\tr A_{m-1} A_m^{-1}
\ge c \tr(A_1^\alpha A_m^{-\alpha}) -d.$$
\end{lem}
\begin{proof} From the above lemma,
$$\tr A_1 A_2^{-1}+\tr A_2 A_3^{-1} \ge 2 \tr(A_1^{1/2}A_3^{-1/2})$$
and
$$\tr A_3 A_4^{-1}+\tr A_4 A_5^{-1} \ge 2 \tr(A_3^{1/2}A_5^{-1/2}).$$
Summing these and applying the lemma again gives
$$\tr A_1 A_2^{-1}+\cdots+\tr A_4 A_5^{-1} \ge 4 \tr(A_1^{1/4}A_5^{-1/4}).$$
Continuing this procedure implies the statement for $m$ odd, in which case
we can take $c=2^{(m-1)/2}$, $\alpha=1/c$ and $d=0$.  For $m$ even, we bound
the last term in the sum by
$$\tr A_{m-1} A_m^{-1} \ge 2^k \tr(A_{m-1}^{1/2^k}A_m^{-1/2^k})-n(2^k-1)$$
with $k=(m-2)/2$, and then applying the lemma again implies the statement
with $c=2^{m/2}$, $\alpha=1/c$ and $d=n(2^k-1)$.

\end{proof}

\begin{lem}\label{lyapunov}
For any integer $k$, 
$$\Delta_X \tr X^k \le \frac12 (k^2+n|k|) \tr X^k.$$
\end{lem}
\begin{proof}
From \eqref{power} we have, for $k$ positive,
$$\Delta_X \tr X^k = \frac{k^2}2 \tr X^k + \frac{k}2 \sum_{j=1}^{k} \tr X^j \tr X^{k-j}.$$
By Chebyshev's sum inequality, $\tr X^j \tr X^{k-j}\le n\tr X^k$ for each $j$. 
Hence
$$\Delta_X \tr X^k \le \frac12 (k^2+nk) \tr X^k.$$
Similarly, using \eqref{power2}, we have for $l$ positive,
$$\Delta_X \tr X^{-l} \le \frac12 (l^2+nl) \tr X^{-l}.$$
\end{proof}

\subsection{Convergence lemma}

\begin{lem}\label{lem-conv}
For any $V,W\in\P$, $\nu,p\in\C$ and $\alpha>0$, the integrals
$$I_1=\int_\P e_\nu(Y) \tr Y^p \etr(- VY^\alpha - WY^{-\alpha}) \mu(dY),$$
and
$$I_2=\int_\P e_\nu(Y) \ln|Y| \etr(- VY^\alpha - WY^{-\alpha}) \mu(dY)$$
are convergent.  
\end{lem}
\begin{proof} Without loss of generality we can assume that
$\nu,p\in\R$.  If $\kappa$ is the smallest among the set of eigenvalues of $V$ and $W$, then 
$$|I_1| \le \int_\P e_\nu(Y)\tr Y^p \etr(-\kappa  [Y^\alpha +  Y^{-\alpha}]) \mu(dY),$$
and
$$|I_2| \le \int_\P e_\nu(Y) |\ln|Y|| \etr(-\kappa  [Y^\alpha +  Y^{-\alpha}]) \mu(dY).$$
These integrals are easily seen to be finite using \eqref{ed} (and in fact can be expressed as
Pfaffians using de Bruijn's formula).
\end{proof}

\subsection{Proof of Proposition \ref{wc}}\label{awc}

First note that, without loss of generality, we can assume $\l\in\R^N$.
For each $1\le i<j\le N$, define
$$\F_{ij}(Y)=\frac1{N-1}\sum_{m=N+i-j}^{N-1} \tr[ Y^m_i(Y^{m+1}_i)^{-1}
+Y^{m+1}_{j+m-N+1}(Y^m_{j+m-N})^{-1}],$$
and observe that
$$\sum_{1\le i<j\le N} \F_{ij}(Y) \le \F(Y).$$
By Lemma \ref{a2}, there exist constants $\alpha,c,d>0$ such that
$$\F_{ij}(Y) \ge c\ \tr[X_j^\alpha (Y^{N+i-j}_i)^{-\alpha} + (Y^{N+i-j}_i)^{\alpha} X_i^{-\alpha}] -d.$$
It follows that
$$\psi_\l(X) \le \prod_{1\le i<j\le N} R_{ij}(X_i,X_j),$$
where
$$R_{ij}(X_i,X_j) = e^d \int_\P e_{\l_{N+i-j}-\l_{N+i-j+1}}(Z) 
\etr\left(-c [ X_j^\alpha Z^{-\alpha} + Z^{\alpha} X_i^{-\alpha}]\right) \mu(dZ).$$
These integrals converge by Lemma~\ref{lem-conv}.  For the second claim,
set $\F_\l(Y)=\F(Y)-\ln e_\l(Y)$.  It follows from the above lower bound
that there exist positive constants $C,D,\alpha$ such that
$$\F(Y) +D \ge C \sum_{1\le i\le m<N} \tr \left[ (X_{N+i-m})^\alpha (Y^m_i)^{-\alpha} 
+ (Y^m_i)^{\alpha} (X_{m+j-N})^{-\alpha} \right] .$$
Now, for any $A,B\in\P$, $a>0$ and $b\in\R$, $\tr(AZ^a+BZ^{-a})+b\ln |Z| \to +\infty$
as $Z\to\partial\P$.  Hence, $\F_\l(Y)\to+\infty$ as $Y\to\partial\T(X)$, as required.

\subsection{Proof of Lemma \ref{duf}}\label{duf-proof}

This follows from the matrix Dufresne identity~\cite[Theorem 1]{rv16}
(Corollary~\ref{mdr} in the present paper),
together with the fact that the 
eigenvalue process of $Y_1(t)^{-1/2} Y_2(t) Y_1(t)^{-1/2}$ has the same law as that of a 
Brownian motion in $\P$ with generator $2\Delta-2\nu\, \tr \vartheta_{X}$ started at $A$.
The latter can be seen as follows.  
Consider a realisation of the process $Y$ defined by $Y_i=G_i^t G_i$, where $G_i$ are 
independent, right-invariant Brownian motions on $GL(n,\R)$, with respective drifts $\l_i$.
Then the eigenvalues of $Y_1(t)^{-1/2} Y_2(t) Y_1(t)^{-1/2}$ are the same as those of 
$X=(G_2 G_1^{-1})(G_2 G_1^{-1})^t$.  The process $X$ is Markov with infinitesimal 
generator $\Sigma=\Delta^{(-\l_1)}+\Omega^{(\l_2)}$, where $\Omega^{(\l_2)}$ is the 
generator of $G_2 G_2^t$, as discussed in Section~\ref{rcp}.  The claim follows from
the fact that $\Sigma$ is $O(n)$-invariant and has the same radial part as 
$\Delta^{(-\l_1)}+\Delta^{(\l_2)}=2\Delta-2\nu\, \tr \vartheta_{X}$.

\bigskip
\noindent{\em Acknowledgements.}
Research supported by the European Research Council (Grant Number 669306).
Thanks to the anonymous referees for many helpful comments on an earlier version,
and to Philippe Bougerol for helpful correspondence.


\begin{thebibliography}{99}


\bibitem{aow19}  
Assiotis, T., O'Connell, N., Warren, J.: Interlacing diffusions. In: Donati-Martin, C., Lejay, A., Rouault, A. (eds.) S\'eminaire de Probabilit\'es L. Lecture Notes in Mathematics, vol. 2252. Springer, Cham (2019)

\bibitem{baudoin} Baudoin, F.:
Further exponential generalization of Pitman's $2M-X$ theorem.  
Electron. Comm. Probab.  {\bf 7},  37--46 (2002)

\bibitem{boc11} Baudoin, F., O'Connell, N.:
Exponential functionals of Brownian motion and class-one Whittaker functions.
Ann. Inst. H. Poincar\'e -- Probab. Statist. {\bf 47}, 1096--1120 (2011)

\bibitem{BBO05} Biane, P., Bougerol, P., O'Connell, N.:
Littelmann paths and brownian paths.
Duke Math. J. {\bf 130}, 127--167 (2005)

\bibitem{bc} Borodin, A., Corwin, I.:  Macdonald processes.
Probab. Theory Related Fields {\bf 158}, 225--400 (2014) 

\bibitem{bcf} Borodin, A., Corwin, I., Ferrari, P.:
Free energy fluctuations for directed polymers in random media in $1+1$ dimension.
Comm. Pure Appl. Math. {\bf 67},  1129--1214 (2014)

\bibitem{bougerol} Bougerol, P.:
The Matsumoto and Yor process and infinite dimensional hyperbolic space.
In: {\em In Memoriam Marc Yor - S\'eminaire de Probabilit\'es XLVII}, Springer, 2015.

\bibitem{bpc} Ph. Bougerol, private communication.

\bibitem{bj} Bougerol, P., Jeulin, T.:
Paths in Weyl chambers and random matrices. 
Probab Theory Relat. Fields {\bf 124}, 517--543 (2002) 

\bibitem{bmrl80} Bruschi M., Manakov S.V., Ragnisco, O., Levi, D.:
The nonabelian Toda lattice: Discrete analogue of the matrix Schr\"odinger spectral problem,
Journal of Mathematical Physics {\bf 21}, 2749 (1980)
https://doi.org/10.1063/1.524393

\bibitem{bfj} Bueno, M.I., Furtado, S., Johnson, C.R.:
Congruence of Hermitian matrices by Hermitian matrices.
Linear Algebra Appl. {\bf 425},  63--76 (2007)

\bibitem{bw}
Butler, R.W.,  Wood, A. T. :
Laplace approximation for Bessel functions of matrix argument.
J. Comput. Appl. Math. {\bf 155}, 359--382  (2003) 

\bibitem{C}
Cerenzia, M.:
A path property of Dyson gaps, Plancherel measures for $Sp(\infty)$, and random surface growth.
Available from arXiv:1506.08742, (2015).

\bibitem{cosz}
Corwin, I., O'Connell, N., Sepp{\"a}l{\"a}inen, T., Zygouras, N.:
Tropical combinatorics and Whittaker functions.
Duke Math. J. {\bf 163},  513--563 (2014)

\bibitem{deb} De Bruijn, N.:
On some multiple integrals involving determinants. 
Journal of the Indian Mathematical Society. New Series, {\bf 19}, 133--151 (1955) 

\bibitem{duf}
Dufresne, D.: The integral of geometric Brownian motion. 
Adv. in Appl. Probab. {\bf 33},  223--241 (2001)

\bibitem{dynkin}
Dynkin, E.B.:
Non-negative eigenfunctions of the Laplace-Beltrami operator and Brownian motion in certain symmetric spaces. 
Dokl. Akad. Nauk SSSR {\bf 141}, 288--291 (1961) 

\bibitem{ek} Ethier, S.N., Kurtz, T.G.:
Markov Processes: Characterization and Convergence. Wiley, New York, 1986.

\bibitem{fw} Fitzgerald, W., Warren, J.:
Point-to-line last passage percolation and the invariant measure of a system of reflecting Brownian motions.
Probab Theory Relat. Fields {\bf 178}, 121--171 (2020)
https://doi.org/10.1007/s00440-020-00972-z
 
\bibitem{g18}
Grabsch, A.: Th\'eorie des matrices al\'eatoires en physique statistique: th\'eorie quantique de la diffusion et syst\`emes d\'esordonn\'es. PhD thesis, Universit\'e Paris-Saclay, 2018.  

\bibitem{gklo}
Gerasimov, A.,  Kharchev, S.,  Lebedev, D., Oblezin, S.:
On a Gauss-Givental representation of quantum Toda chain wave equation. 
Int. Math. Res. Not. (2006) 1--23.

\bibitem{herz} Herz, C.S.:
Bessel Functions of Matrix Argument.
Ann. Math. {\bf 61},  474--523 (1955)

\bibitem{is} Imamura, T., Sasamoto, T.:
Determinantal structures in the O'Connell-Yor directed random polymer model.
J. Stat. Phys. {\bf 163},  675--713 (2016)

\bibitem{kelly} Kelly, F.P.: Markovian functions of a Markov chain. 
Sankya Ser A {\bf 44},  372--379 (1982)

\bibitem{ks1} Kemeny, J.G., Snell, J.L.:  {\em Finite Markov Chains.} Van Nostrand, Princeton, 1960.

\bibitem{kurtz} Kurtz, T.G.:
Martingale problems for conditional distributions of Markov processes.
Electron. J. Probab. {\bf 3}, 1--29 (1998)

\bibitem{k11} Kurtz, T.G.:
Equivalence of Stochastic Equations and Martingale Problems.
In: Stochastic Analysis 2010, D. Crisan (ed.), Springer, 2011.

\bibitem{lnr} Liechty, K., Nguyen, G.B., Remenik, D.:
Airy process with wanderers, KPZ fluctuations, and a deformation of the Tracy-Widom GOE distribution.
arXiv:2009.07781

\bibitem{mj}
Matheny, D., Johnson, C. R. : Congruence of Hermitian matrices by Hermitian matrices.
William and Mary NSF-REU report, 2005.

\bibitem{my}  Matsumoto, H., Yor, M.:
A version of Pitman's $2M-X$ theorem for geometric Brownian motions.
C. R. Acad. Sci. Paris {\bf 328},  1067--1074 (1999)

\bibitem{nr17}
Nguyen, G. B.,  Remenik, D.: 
Non-intersecting Brownian bridges and the Laguerre orthogonal ensemble. 
Ann. Inst. H. Poincar\'e -- Probab. Statist. {\bf 53}, 2005--2029 (2017).

\bibitem{nomura} Nomura, T.:
Algebraically independent generators of invariant differential operators on a symmetric cone.
Journal fur die Reine und Angewandte Mathematik {\bf 400}, 122--133 (1989) 
https://doi.org/10.1515/crll.1989.400.122

\bibitem{nrw}
Norris, J. R., Rogers, L. C. G., Williams, D.:
Brownian motions of ellipsoids.
Trans. Amer. Math. Soc. {\bf 294},  757--765 (1986)

\bibitem{noc12}
O'Connell, N.:  Directed polymers and the quantum Toda lattice.
Ann. Probab. {\bf 40},  437--458 (2012)

\bibitem{noc13} O'Connell, N.:  Geometric RSK and the Toda lattice.   
Illinois J. Math. {\bf 57}, 883--918 (2013)

\bibitem{noc14} O'Connell, N.:
Whittaker functions and related stochastic processes. 
In: Random matrices, interacting particle systems and integrable systems, MSRI Vol. 65, 2014.

\bibitem{noc15} O'Connell, N.:  Stochastic B\"acklund transformations. 
In: In Memoriam Marc Yor - S\'eminaire de Probabilit\'es XLVII, Springer, 2015.

\bibitem{osz}
O'Connell, N., Sepp{\"a}l{\"a}inen, T., Zygouras, N.:
Geometric RSK correspondence, Whittaker functions and symmetrized random polymers.
Invent. Math. {\bf 197}, 361--416 (2014) 

\bibitem{oy1}
O'Connell, N., Yor, M.:
Brownian analogues of Burke's theorem. 
Stoch. Process. Appl. {\bf 96},  285--304 (2001)

\bibitem{pop81} Popowicz, Z.:  Some remarks about the lattice chiral field,
Phys. Lett. A {\bf 81}, 235--236 (1981)

\bibitem{pop83} 
Popowicz, Z.: The generalized non-abelian Toda lattice,
Z. Phys. C - Particles and Fields {\bf 19}, 79--81 (1983)
https://doi.org/10.1007/BF01572340

\bibitem{rv16}
Rider, B., Valk\'o, B.:  Matrix Dufresne identities.
Int. Math. Res. Not. (2016) 174--218.

\bibitem{rp} Rogers, L.C.G., Pitman, J.W.:  Markov functions. Ann. Prob. {\bf 9}, 573--582 (1981) 

\bibitem{sv}  Sepp{\"a}l{\"a}inen, T., Valk\'o, B.:
Bounds for scaling exponents for a $1+1$ dimensional directed polymer in a Brownian environment.
Alea {\bf 7}, 451--476 (2010) 

\bibitem{spohn} Spohn, H.:
KPZ scaling theory and the semidiscrete directed polymer model.
In: Random matrices, interacting particle systems and integrable systems, MSRI Vol. 65, 2014.

\bibitem{stade}
Stade, E.:
Archimedean $L$-factors on $GL(n) \times GL(n)$ and generalized Barnes integrals. 
Israel J. Math. {\bf 127}, 201--219 (2002) 

\bibitem{terras} Terras, A.:
Harmonic Analysis on Symmetric Spaces, Volume 2.
Second edition, Springer 2015.

\bibitem{wl} Wang, Z.-L., Li, S.-H.:  BKP hierarchy and Pfaffian point process.
Nuclear Physics B {\bf 939},  447--464 (2019)

\bibitem{wonham} Wonham, W.:  On a matrix Riccati equation of stochastic control.
SIAM J. Control {\bf 6},  681--697 (1968)

\end{thebibliography}
\end{document}